\documentclass[11pt, a4paper]{article}

\title{Onsager--Machlup functional for multiple SLEs 
}
\date{\today}

\usepackage{comment}

\usepackage[T1]{fontenc}

\usepackage[latin9]{inputenc}

\usepackage{mathtools}
\usepackage{lmodern}
\usepackage{amsthm,amsmath,amsfonts,amssymb}
\usepackage{graphicx}
\usepackage{enumerate,enumitem}
\usepackage{authblk}
\usepackage{cite,url}
\usepackage[normalem]{ulem}
\usepackage{mathrsfs}
\usepackage{bbm}

\usepackage{multirow}
\usepackage{array}
\newcolumntype{P}[1]{>{\centering\arraybackslash}p{#1}}
\usepackage{longtable,booktabs}
\usepackage{extarrows}
\usepackage[all]{xy}

\usepackage{hyperref}
\hypersetup{
    colorlinks=true, 
    linkcolor=blue, 
    urlcolor=black, 
    citecolor=black,
    linktoc=all 
}

\usepackage[font=normal,labelfont=bf,labelsep=quad]{caption}

\usepackage[activate={true,nocompatibility},final,tracking=true,kerning=true,spacing=false,factor=1100,stretch=20,shrink=20]{microtype}
\microtypecontext{spacing=nonfrench}

\allowdisplaybreaks

\setlist[enumerate]{topsep = 1ex, leftmargin=1cm, itemsep= -2pt}
\setlist[itemize]{topsep = 1ex, leftmargin=1cm, itemsep= -2pt}

\let\OLDthebibliography\thebibliography
\renewcommand\thebibliography[1]{
  \OLDthebibliography{#1}
  \setlength{\parskip}{1pt}
  \setlength{\itemsep}{2pt}
}

\newtheorem{thm}{Theorem}[section]

\newtheorem{lem}[thm]{Lemma}

\theoremstyle{definition}

\newtheorem{remark}[thm]{Remark}

\newcommand{\Schwarzian}{\mathcal{S}}

\numberwithin{equation}{section}

\usepackage{floatrow}

\usepackage{mathtools}

\usepackage[dvipsnames]{xcolor}

\global\long\def\ii{\mathfrak{i}}

\newcommand{\abs}[1]{\left\lvert #1 \right \rvert}

\newcommand{\mc}[1]{\mathcal{#1}}
\newcommand{\m}[1]{\mathbb{#1}}

\makeatletter
\newcommand{\Rm}[1]{\expandafter\@slowromancap\romannumeral #1@}
\makeatother

\renewcommand\Re{\operatorname{Re}}

\def\K{K}

\def\SLE{\operatorname{SLE}}

\def\Mo{M\"obius }

\def\a{\alpha}
\def\b{\beta}
\def\g{\gamma}

\def\d{\delta}

\def\e{\epsilon}

\def\t{\theta}

\def\k{\kappa}

\def\s{\sigma}

\def\o{\omega}
\def\O{\Omega}

\def\SC{{\mc {SC}}}

\def\dd{\mathrm{d}}

\def\P{\m P}

\def\1{\mathbf{1}}

\author{Shuo Fan \thanks{\protect\url{shuofan.math@gmail.com} Tsinghua University, China; IHES, France} }

\begin{document}

\maketitle
\begin{abstract}
  Recently, Carfagnini and Wang established that the loop Loewner energy can be interpreted as the Onsager--Machlup functional for the SLE loop measure \cite{CW24}. In this paper, we first interpret the multi-chordal Loewner potential as an Onsager--Machlup functional for the multi-chordal SLE. Subsequently, we extend the conformal deformation formula to the multi-radial Loewner potential and derive the Onsager--Machlup functional for the multi-radial SLE.
\end{abstract}
\tableofcontents
\section{Introduction}
\subsection{Background and main results}
At the turn of the millennium, Oded Schramm introduced the Schramm--Loewner evolution (SLE), a one-parameter family of random planar fractal curves generated by Brownian motion through Loewner transform \cite{Sch00,RS05}.
The loop version of SLE was introduced and studied in \cite{werner_measure,Benoist_loop,zhan2020sleloop} and is also known as Malliavin--Kontsevich--Suhov measure \cite{KS07,bj_cft_loop}. When the parameter $\kappa \in (0,4]$, the SLE$_\kappa$ loop measure $\nu_\k$ is supported on the space of Jordan curves.
The loop Loewner energy was introduced in \cite{RW} as a M\"obius invariant function on the family of Jordan curves, which was later shown to be closely related to the geometry of universal Teichm\"uller space \cite{W2}, random conformal geometry \cite{VW1,VW2}, and hyperbolic renormalized volume \cite{bbvw}. 

The Onsager--Machlup functionals originate from \cite{MO53,OM53} and compare the probabilities of a diffusion process staying in two infinitesimal neighborhoods. Recently, \cite{CW24} showed that the Onsager--Machlup functional of the SLE loop is exactly a multiple of the Loewner energy. More precisely,
\begin{equation*}
    \lim_{\e \rightarrow 0} \frac{\nu_\k(O_\e(\g))}{\nu_\k(O_\e(\m S^1))}=\exp\left(\frac{c(\k)}{24}I^L_{\text{loop}}(\g)\right)
\end{equation*}
where $c(\k)=(6-\k)(3\k-8)/(2\k)$ is the central charge of SLE$_\kappa$, $\m S^1$ is the unit circle, $\g$ is an analytic Jordan curve, $O_\e(\g)$ denotes an admissible neighborhood of $\g$ with size $\e$ and $I^L_{\text{loop}}$ denotes the loop Loewner energy.  
The key ingredient of the proof is the conformal restriction covariance property, which uniquely determines the SLE loop measure \cite{KS07,bj_cft_loop,CG25_CR}.

In this work, we investigate the Onsager--Machlup functional for other variants of SLE.
We develop a comprehensive framework for interpreting  Loewner potential (defined later) as the Onsager--Machlup functional for SLEs. We consider the following cases:
\begin{itemize}
    \item Single chordal SLE (Theorem~\ref{OM for single chordal}),
    \item Multiple chordal SLE (Theorem~\ref{OM for multiple chordal}),
    \item  Chordal SLE with force points (Theorem~\ref{OM for chordal variants}),
    \item Radial SLE (Theorems~\ref{OM for single radial},~\ref{OM for radial variants},~\ref{OM for multiple radial}).
\end{itemize}

To state the result, let us first describe the setup, including the choices of $\s$-algebra, admissible neighborhoods, and SLE measures.

We choose first a reference configuration.
For that, we fix a simply connected bounded domain $D$ with an analytic boundary (e.g., the unit disk $\m D$), two non-negative integers $n_1$ and $n_2 $, $n_1$ interior points and $n_2$ boundary points $\bar x_0=(x_{0,j})_{j=1}^n$ with $n=n_1+n_2$.
In the multi-chordal case, we have $n_1=0$ and $n_2$ is even. Let $\mc X_0$ denote the space of the simple, disjoint, multi-chord $ \g$ in $D$ connecting the boundary points pairwise with an arbitrarily chosen link pattern $\a_0$, while the $\s$-algebra $\mc F_0$ is induced by the Hausdorff metric. In the multi-radial case, we have $n_1=1$ and $n_2 \ge 1$. Let $\mc X_0$ denote the space of the simple multi-arc $ \g$ in $D$ connecting the interior point to each boundary point, while each arc is disjoint from the others except at the target interior point. The $\s$-algebra $\mc F_0$ is induced by the Hausdorff metric. 

Once the reference configuration is chosen, for other link pattern $\a$, $n_1$ interior points and $n_2$ boundary points $\bar x$, we define the space of multicurves $\mc X = \mc X_{\a, \bar x}$ and corresponding $\sigma$-algebra $\mc F = \mc F_{\a, \bar x}$ similarly. 

For admissible neighborhoods, we first fix the reference element $\g_0 \in \mc X_0$ to be the unique minimizer of the Loewner potential in $\mc X_0$ defined in \eqref{eq single chordal potential}\eqref{eq multiple chordal potential}\eqref{eq forced chordal potential}\eqref{eq single radial potential}\eqref{eq forced radial potential}\eqref{eq multiple radial potential} and a decreasing family of neighborhoods $(A_\e)_{\e >0}$ of $\g_0$ in $D$. Then for any element $\g \in \mc X$ that is locally conformally equivalent to $\g_0$ in $D$, which means that there exists a conformal map $f:A_{\e_0} \rightarrow \tilde A_{\e_0}$ for some $\e_0$ such that $\g=f(\g_0)$ and define $\tilde A_\e=f(A_\e)$ for $\e < \e_0$. 
We require that those neighborhoods $A_\e$ and $\tilde A_\e$ be subsets of $D$ and coincide with the regular boundary of $D$ near the boundary points and contain the interior points.
Let us define the admissible neighborhoods of $\gamma_0$ and $\gamma$ as
\begin{align*}
    &O_\epsilon(\gamma_0):=\{\eta \in \mc X_0~|~\eta \subset A_\e\},\\
    &O_\epsilon(\gamma):=\{\eta \in\mc X~|~\eta \subset \tilde A_\e\}.
\end{align*}
\begin{lem}\label{lem Hausdorff topology}
    The admissible neighborhoods $O_\e$ form a neighborhood basis for the Hausdorff topology on $\mc X$.
\end{lem}

In this paper, we refer to the SLE measure as the product of the SLE partition function and the SLE probability measure.
The partition functions (see, e.g.,\cite{BBK,BD06_growth,Dub_couplings,Dub_comm,Law09,Dub_duality,Zhan,FK_CFT}) are smooth positive functions that determine the SLE measures. They satisfying some hypoelliptic partial differential equations giving rise to SLE martingales and M\"obius covariance, a consequence of the conformal invariance from statistical physical models \cite{Smi:ICM,SS05_explorer}.
For $\k \in (0,4]$, let $Q_{D,\bar x}^\k$ denote the corresponding SLE measure on $(\mc X, \mc F)$. We fix the reference $\SLE_\k$ measure $Q_{D,\bar x_0}^\k$ on $(\mc X_0, \mc F_0)$.

\begin{thm}\label{main theorem}
    For $\k \in (0,4]$, for any element $\g \in\mc X$ that is locally conformally equivalent to $\g_0 \in \mc X_0$ with a collection of admissible neighborhoods $O_\e(\g)\in \mc F$ inherited via a conformal map $f$ from the reference element $\g_0 \in \mc X_0$ with $O_\e(\g_0)\in \mc F_0$ defined as above, let $Q_{D,\bar x_0}^\k$ and $ Q_{D,\bar x}^\k$ denote the $\SLE$ measure on the measurable space $(\mc X_0, \mc F_0)$ and $(\mc X, \mc F)$, respectively, then we have
    \begin{equation}\label{eq OM}
          \lim_{\epsilon \rightarrow 0} \frac{ Q_{D,\bar x}^\kappa(O_\epsilon({\gamma}))}{ Q_{D,\bar x_0}^\kappa(O_\epsilon({\gamma}_0))}=\exp \left( \frac{c(\kappa)}{2} (\mc H_{D,\bar x}({\gamma}) -\mc H_{D,\bar x_0}({\gamma}_0)) -F_\k(\g)\right),
    \end{equation}
    where $\mc H$ denotes the Loewner potential defined in \eqref{eq single chordal potential}\eqref{eq multiple chordal potential}\eqref{eq forced chordal potential}\eqref{eq single radial potential}\eqref{eq forced radial potential}\eqref{eq multiple radial potential} depending on the type of SLE, and the function
    \begin{equation*}
    F_\k(\g)=\sum_{j=1}^n e_\k(j)\log \abs{f'(x_{0,j})},
    \end{equation*}
   for some explicit number $e_\k(j) \in \m R$. Moreover, $F_\k(\g)$ depends on $\g_0$, the admissible neighborhoods, the conformal map $f$, but we have \begin{equation}\label{eq kF}
    \lim_{\k \to 0}\k F_\k(\g)= 0.
\end{equation}
\end{thm}
The multi-chordal Loewner potential was introduced in the large deviation principle of multi-chordal SLE, as developed in the foundational work of \cite{PW24_LDP_multichordal}. Subsequent research has explored various extensions and formulations and the associated large deviation principles, notably in \cite{AOP_LDP_multiradial,Kru_rho_energy,AP_ADP_capacity,HPW_multiradial_LDP}.
\subsection{Strategy and discussion}
Following the strategy of \cite{CW24}, the elementary argument of Theorem~\ref{main theorem} proceeds through an application of the conformal covariance and generalized conformal restriction property (alternatively called boundary perturbation). These crucial properties are inherent to SLE processes \cite{Friedrich_Werner_03,WW_survey_CR,HPW_multiradial_perturbation,Wu15_conformal_restriction_radial,JL18,Wei18_conformal_restriction_trichordal,CG25_CR}, which we summarize below.
See Lemma~\ref{lem property of single chordal},~\ref{lem property of forced chordal},~\ref{lem property of single radial},~\ref{lem property of forced radial},~\ref{lem property of multiple radial}.
\begin{lem}\label{lem boundary perturbation and conformal convariance}
    Suppose $D \subset D' \subsetneqq \m C$ are simply connected domains. Suppose that $\bar x=(x_j)_{j=1}^n$ are $n$ distinct boundary or interior points of $\bar D$ and $\bar D'$. And assume that $\partial D$ and $\partial D'$ are analytic and agree in the neighborhoods of the boundary points. Let $Q^\k_{D, \bar x}$ denote the $\SLE_\k$ measure of a fixed type in $D$ with $\bar x$ as boundary and interior points, then we have
    \begin{itemize}
        \item ${\operatorname{(COV)}}$ For any conformal map $f$ on $D$, we have the conformal covariance rule: there exist $\bar b_\k=(b_\k(j))_{j=1}^n$
            \begin{equation*}
                f \circ Q^\kappa_{D; \bar x} = \abs{f'(\bar x)}^{\bar b_\k}  Q^\kappa_{f(D); f(\bar x)}.
            \end{equation*}
We use the notation $\abs{f'(\bar x)}^{\bar b(\kappa)}  =\prod_{j=1}^n \abs{f'( x_j)}^{\bar b_\kappa(j)}  $.
        \item $\operatorname{(GCR)}$ The law of $Q^\kappa_{D; \bar x}$ is absolutely continuous with respect to $Q^\kappa_{D'; \bar x}$ with Radon--Nikodym derivative 
            \begin{equation*}
                \mathbbm{1}_{\{\gamma \subset D\}} Y_{D,D';\bar x}(\gamma) = \mathbbm{1}_{\{\gamma \subset D\}} \exp\left( \frac{c(\kappa)}{2} \mc B(\gamma, D'\backslash D; D')\right),
            \end{equation*}
    \end{itemize}
    where $\mc B(\gamma, D'\backslash D; D')$ denotes the total mass of the set of Brownian loops that stay in $ D'$ and intersect both $\g$ and $ D' \backslash D$ (plus the cross terms in the multi-chordal case).
      
\end{lem}
Through this framework, we establish Theorem~\ref{main theorem} via systematic analysis of the conformal deformation of the Loewner potential, which is inspired by \cite{RW,SW24_deforamtion,Wang21_note_on_conformal_restrcition}, especially the conformal deformation of the chordal Loewner energy. The technical core resides in the following lemmas:
\begin{itemize}
    \item Lemma~\ref{conformal deformation of Loewner potential} addresses conformal deformation of the chordal Loewner potential,
    \item Lemma~\ref{conformal deformation of rho-Loewner potential} extends these considerations to variants with force points,
    \item Lemma~\ref{conformal deformation of radial Loewner energy} establishes radial analogues,
    \item Lemma~\ref{conformal deformation of multi-radial Loewner potentail} establishes multi-radial analogues.
\end{itemize}
\begin{thm}\label{thm 2}
    For an element $\g \in \mc X$ with finite Loewner potential $\mc H_{D,\bar x}(\g)$, let $A$ be a neighborhood of $\g$ that agrees with $D$ near the boundary or interior points $\bar x=(x_j)_{j=1}^n$. Assume $f$ is a conformal map on $A$ such that $f(\g)$ is in $\mc X_{D,f(\bar x)}$ and $f(A)\subset D$ agrees with $D$ (including the boundary) near the boundary points $(f(x_j))_{j=1}^n$. Then we have
    \begin{equation*}
        \mc H_{D;f(\bar x)}({f(\gamma)})- \mc H_{D;\bar x}(\gamma)=\mc B(\gamma, D\backslash A;  D)- \mc B(f(\gamma),  D\backslash f(A);  D)+\sum_{j=1}^n e(j)\log \abs{f'(x_j)},
    \end{equation*}
    where
    \begin{equation*}
      e(j)=\lim_{\k \to 0}- \frac{2b_\k(j)}{c(\k)}.
    \end{equation*}
\end{thm}
\begin{remark}
    Although $e_j$ equals the limit as $\k \to 0$, we prove Theorem~\ref{thm 2} using a deterministic approach.
    It may be viewed as a deterministic version of the generalized conformal restriction.
\end{remark}
\begin{proof}[Proof of Theorem~\ref{main theorem} using Lemma~\ref{lem boundary perturbation and conformal convariance} and Theorem~\ref{thm 2}]
    Fix $\epsilon_0$, we set $A=A_{\epsilon_0}$, $\tilde{A}=f(A_{\epsilon_0})$ and $\bar x = f(\bar x_0)$. 
    By conformal restriction and conformal covariance, we have
    \begin{align*}
        Q_{D; \bar x}^\kappa(O_\epsilon(\gamma))&\xlongequal{\text{Def}}\int \mathbbm{1}_{\{\tilde{\eta} \subset \tilde{A_\epsilon}\}} \dd Q_{D;\bar x}^\kappa (\tilde{\eta})\xlongequal{\text{GCR}} \int \mathbbm{1}_{\{\tilde{\eta} \subset \tilde{A_\epsilon}\}} \left(Y_{\tilde{A},D;\bar x}(\tilde{\eta})\right)^{-1} \dd Q^\kappa_{\tilde{A}; \bar x} (\tilde{\eta})\\
        &\xlongequal{\text{COV}} \int \mathbbm{1}_{\{\eta \subset A_\epsilon\}} \left(Y_{\tilde{A},D;\bar x}(f(\eta))\right)^{-1}\abs{f'(\bar x_0)}^{\bar b_\kappa} \dd Q^\kappa_{A; \bar x_0} (\eta) \\
        &\xlongequal{\text{GCR}} \int \mathbbm{1}_{\{\eta \subset A_\epsilon\}} \frac{Y_{A, D ; \bar x}(\eta)}{Y_{\tilde{A},D;\bar x}(f(\eta))} \abs{f'(\bar x_0)}^{\bar b_\kappa} \dd Q_{D; \bar x_0}^\kappa (\eta).
  \end{align*}
        
Using Lemma~\ref{lem boundary perturbation and conformal convariance}, we have that 
  \begin{align*}
      &\frac{Y_{A, \m H ; 0 ,\infty}(\eta)}{Y_{\tilde{A},\m H;x,y}(f(\eta)} \abs{f'(\bar x_0)}^{\bar b_\kappa}\\
      &= \exp \left( \frac{c(\kappa
      )}{2} ( \mc B(\eta, \m H \backslash A; \m H) -  \mc B(f(\eta), \m H \backslash \tilde{A}; \m H) ) -\sum_{j=1}^nb_\kappa(j) \log\abs{f'(x_{0,j})}\right) .
  \end{align*}
Using the uniform convergence of Brownian loop measure (see Lemma~\ref{1uniform convergence}, ~\ref{2uniform convergence}, ~\ref{3 uniform convergence}), we have
\begin{align*}
    \frac{Q_{D; \bar x}^\kappa(O_\epsilon(\gamma))}{Q_{D; \bar x_0}^\kappa(O_\epsilon(\gamma_0))}
    &=\frac{  \int \mathbbm{1}_{\{\eta \subset A_\epsilon\}} \frac{Y_{A, D ; \bar x}(\eta)}{Y_{\tilde{A},D;\bar x}(f(\eta))} \abs{f'(\bar x)}^{\bar b_\kappa} \dd Q_{D; \bar x_0}^\kappa (\eta)}{ \int \mathbbm{1}_{\{\eta \subset A_\epsilon\}}  \dd Q_{D; \bar x_0}^\kappa (\eta)}\\
    &\stackrel{\epsilon \rightarrow 0+}{\longrightarrow}
    \exp \left( \frac{c(\kappa)}{2}(\mc B(\gamma_0, D \backslash A; D) -  \mc B(\gamma, D \backslash \tilde{A}; D))-\sum_{j=1}^nb_\kappa(j) \log\abs{f'(x_{0,j})}\right).
\end{align*}
By Theorem~\ref{thm 2}, the item in the above exponent equals
\begin{align*}
     \frac{c(\kappa)}{2}\left(\mc H_{D;\bar x}(\gamma)-\mc H_{D;\bar x_0}(\gamma_0)\right)-F_\k(\g),
\end{align*}
where
\begin{equation*}
        F_\k(\g)=\sum_{j=1}^n e_\k(j)\log \abs{f'(x_{0,j})},
    \end{equation*}
    with
     \begin{equation*}
        e_\k(j)=b_\k(j)+\frac{c(\k)}{2}e(j)=b_\k(j)-c(\k)\lim_{\k \to 0} \frac{b_\k(j)}{c(\k)}.
    \end{equation*}
    It is not hard to check that $\k e_\k(j)\to 0$ hence $\k F_\k(\g) \to 0$ as $\k \to 0$.
\end{proof}

Now we discuss our result and relate it with the large deviation principle. We interpret the multiple Loewner potential rather than the Loewner energy as an Onsager--Machlup functional for the multiple SLE, as the two SLE measures that are of the same type in the fraction
\[
\frac{ Q_{D,\bar x}^\kappa(O_\epsilon({\gamma}))}{ Q_{D,\bar x_0}^\kappa(O_\epsilon({\gamma}_0))}
\]
may have different boundary points and different link patterns. This can be understood in the following way: we are considering the Onsager--Machlup functional of a generalized version of SLE measure, which is a measure not only on the multi-chords or the multi-arcs but also on the boundary points.
For the probability measure $\m P^\k=Q_{D,\bar x_0}^\k/\abs{Q_{D,\bar x_0}^\k}$, an element $\g \in \mc X_0$ that is locally conformally equivalent to $\g_0 \in \mc X_0$ and the Loewner energy 
\begin{equation*}
    I^L(\g)=I^L_{D,\bar x_0}(\g)=\left(\mc H_{D,\bar x_0}({\gamma}) -\mc H_{D,\bar x_0}({\gamma}_0)\right)/12,
\end{equation*}
\eqref{eq OM} turns out to be
 \begin{equation}\label{eq OM special}
      \lim_{\epsilon \rightarrow 0} \frac{ \P^\kappa(O_\epsilon({\gamma}))}{ \P^\kappa(O_\epsilon({\gamma}_0))}=\exp \left( \frac{c(\kappa)}{24} I^L(\g) -F_\k(\g)\right),
\end{equation}
which is compatible with the large deviation principle (LDP) as the limit order is commutative

\begin{displaymath}
\xymatrix
{
    -\k\log {\P^\kappa(O_\epsilon({\gamma}))}+\k \log{ \P^\kappa(O_\epsilon({\gamma}_0))} \ar[d]_{\k \to 0}^{\text{LDP}} \ar[r]^{~~~~~~\e\to 0}_{~~~~~~\eqref{eq OM special}}
   &-\k \left( \frac{c(\kappa)}{24} I^L(\g) -F_\k(\g)\right)\ar[d]^{\k \to 0}_{\eqref{eq kF}} \\
    \approx 
    \inf_{\eta \in O_\epsilon({\gamma})} I^L(\eta)- \inf_{\eta \in O_\epsilon({\gamma_0})} I^L(\eta)  \ar[r]_{~~~~~~~~~~~~~~~~~~\e \to 0}^{~~~~~~~~~~~~~~~~~~~~\text{lower-semicontinuity}} &I^L(\g). }
\end{displaymath}

The precise meaning of  ``$\approx$'' in the large deviation principle is as follows:
\begin{equation*}
    {\inf}_{\eta \in \overline{O_\epsilon({\gamma})}} I^L(\eta)\leq\lim_{\k \to 0}-\k\log {\P^\kappa(O_\epsilon({\gamma}))}\leq  {\inf}_{\eta \in O_\epsilon({\gamma})} I^L(\eta).
\end{equation*}
\begin{equation*}
      {\inf}_{\eta \in {O_\epsilon({\gamma_0})}}\leq\lim_{\k \to 0}-\k \log{ \P^\kappa(O_\epsilon({\gamma}_0))}\leq   {\inf}_{\eta \in \overline{O_\epsilon({\gamma_0})}}I^L(\eta).
\end{equation*}
Nevertheless, the lower and upper bounds tend to the same limit as $\e \to 0$ using the lower-semicontinuity of Loewner energy.
Finally we remark that, using Lemma~\ref{lem arbitary choice of F}, we could have
\begin{equation}\label{eq OM special special}
          \lim_{\epsilon \rightarrow 0} \frac{ \P^\kappa(O_\epsilon({\gamma}))}{ \P^\kappa(O_\epsilon({\gamma}_0))}=\exp \left( \frac{c(\kappa)}{24} I^L(\g) \right),
    \end{equation}
    with a special choice of the conformal map $f$. However, the choice is not always independent of $\k$ unless there exists $c(j)$ such that $b_\k(j)=b_\k(1)c(j)$ for $1\leq j\leq n$. Besides, even when the chords $\g_0$ and $\g$ are hyperbolic geodesics in the unit disk $\m D$, the special conformal map $f$ does not coincide with any natural M\"obius map $m: \m D\rightarrow \m D$ satisfying $m(\g_0)=\g$.

\vspace{8pt}

\textbf{Acknowledgements.} The author wishes to thank Yilin Wang for her helpful discussions and generous guidance. 
The author thanks Chongzhi Huang, Eveliina Peltola and Hao Wu for sharing their manuscript \cite{HPW_multiradial_perturbation,HPW_multiradial_LDP} and helpful discussion.

This work was completed in part and discussed during the complex analysis conference ``Complex Faces'' in Venice. The author thanks the organizers for their preparation and hospitality, and thanks the participants for their contribution. The author is funded by Beijing Natural Science Foundation (JQ20001); the European Union (ERC, RaConTeich, 101116694) and Tsinghua scholarship for overseas graduates studies (2023076).

\section{Preliminaries}
Below, we assume any $D$ and $D^\prime$ are simply connected domains with non-polar boundaries, which are proper subsets of $\m C$. When we talk about the boundary point, we assume the boundary is analytic around it.
\subsection{Brownian loop measure}
The Brownian loop measure from \cite{LW2004loupsoup} satisfies the following two properties:
\begin{itemize}
    \item (Restriction property) If $D' \subset D$, then $\dd \mu^{\textbf{BL}}_{D'}(\cdot)=\mathbbm{1} _{\{\cdot \subset D'\}} \dd \mu^{\textbf{BL}}_{D}(\cdot)$.
    \item (Conformal invariance)  If $D$ and $D'$ are conformally equivalent, then the pushforward of $\mu^{\textbf{BL}}_{D}$ via any conformal map from $D$ to $D'$, is exactly $\mu^{\textbf{BL}}_{D'}$.
\end{itemize}
Note that the total mass of loops contained in $\m C$ is infinite. However, when $D \subsetneqq \m C$ is a simply connected domain with a non-polar boundary, and $V_1$ and $V_2$ are two disjoint compact subsets of $D$, the total mass of the set of loops $\mc L(V_1, V_2;D)$ that do stay in $D$ and intersect both $V_1$ and $V_2$ is finite. We then define
\begin{equation*}
    \mc B(V_1, V_2; D):=\mu^{\textbf{BL}}_D(\mc L(V_1, V_2;D)).
\end{equation*}
\subsection{Conformal deformation of chordal Loewner potential}
Let $\gamma$ be a simple chord in $(\m H;0, \infty)$, which we choose to parameterize by the half-plane capacity seen from $\infty$. That is, the conformal map $g_t: \m H \backslash \gamma[0, t] \rightarrow \m H$ can be normalized by $g_t(z)=z+2t/z+ o(1/z)$ near $\infty$. By extension, we can define $W(\cdot)=g_\cdot(\gamma(\cdot))$, which is called the driving function of $\gamma$. The chord $\g$ can be recovered from $W$ using the Loewner equation
\begin{equation*}
    \partial_t g_t(z)=\frac{2}{g_t(z)-W_t},
\end{equation*}
with $g_0(z)=z$.
The chordal Loewner energy of $\gamma$ in  $(\m H; 0, \infty)$ is defined by 
\begin{equation*}
    I_{ \m H;0, \infty}(\gamma):=I(W):=\frac{1}{2}\int_0^\infty (\partial_tW(t))^2\dd t
\end{equation*}
when $W$ is absolutely continuous and is $\infty$ otherwise. For any simply connected domain $D \subsetneqq \m C$ with two prime ends $x$ and $y$, using a conformal map $\psi:D \rightarrow \m H$ such that $\psi(x)=0$ and $\psi(y)=\infty$, we can define the chordal Loewner energy of $\gamma$ in $(D;x,y)$ by
\begin{equation*}
    I_{D;x,y}(\gamma):=I_{\m H; 0, \infty}(\psi(\gamma)),
\end{equation*}
which does not depend on the choice of $\psi$.
The Loewner potential is defined by
\begin{equation}\label{eq single chordal potential}
     \mc H_{D; x,y}({\gamma}):= \frac{1}{12}  I_{D; x,y}(\gamma)  -\frac{1}{4}\log P_{D; x,y},
 \end{equation}
 where $P_{D; x,y}$ is the Poisson excursion kernel corresponding to $H^\k_{D;x,y}$ with $\k=2$ in the next subsection.
As $x$ and $y$ are the boundary points of $\g$, we write $I_D(\g)$ or $\mc H_D(\g)$ for short and use $I(\g)$ or $\mc H(\g)$ when $D=\m H$.
For a compact $\m H-$hull $K$ at positive distance to $0$, from \cite{Wang21_note_on_conformal_restrcition} , the chordal Loewner energy in $(\m H;0, \infty)$ and in $(\m H \backslash K;0, \infty )$ differ by 
\begin{equation*}
    I_{\m H \backslash K;0, \infty }(\gamma)-I_{\m H;0, \infty}(\gamma)= 3 \log\abs{\psi'(0)\psi'(\infty)} +12 \mc B(\gamma, K; \m H),
\end{equation*}
where $\psi : \m H \backslash K \rightarrow \m H$ is conformal and fixes $0$, $\infty$. Using this, we could obtain the conformal deformation of the Loewner potential as follows.
\begin{lem}\label{conformal deformation of Loewner potential}
    Let $\g$ be a chord in $\m H$ with finite Loewner energy and $A$ be a neighborhood of $\g$ in $\m H$ that agrees with $\partial \m H$ near the boundary points $x$ and $y$ of $\g$. Assume $f$ is a conformal map on $A$ such that $f(\g)$ is also a chord in $\m H$ and $f(A)$ agrees with $\partial \m H$ near the boundary points of $f(\g)$. Then we have 
    \begin{equation*}
        \mc H_{\m H;f(x), f(y)}({f(\gamma)})- \mc H_{\m H;x, y}(\gamma)= \frac{1}{4}\log\abs{f'(x)f'(y)}+ \mc B(\gamma, \m H\backslash A; \m H)- \mc B(f(\gamma), \m H\backslash f(A); \m H).
    \end{equation*}
\end{lem}
\begin{proof}
    As $\m H\backslash A$ and $\m H\backslash f(A)$ are compact $\m H-$hulls at positive distance to the boundary points of $\g$ and $f(\g)$ respectively, we have
    \begin{equation*}
    I_{A;x, y }(\gamma)-I_{\m H;x, y}(\gamma) = 3 \log\abs{\psi_1'(x)\psi_1'(y)} +12 \mc B(\gamma, \m H\backslash A; \m H),
\end{equation*}
where $\psi_1 : A \rightarrow \m H$ is conformal and fixes $x$ and $y$, and 
\begin{equation*}
    I_{f(A);f(x),f( y) }(f(\gamma))-I_{\m H;f(x), f(y)}(f(\gamma)) = 3 \log\abs{\psi_2'(f(x))\psi_2'(f(y))} +12 \mc B(f(\gamma), \m H\backslash f(A); \m H),
\end{equation*}
where $\psi_2 : f(A) \rightarrow \m H$ is conformal and fixes $f(x)$ and $f(y)$.
Note that $I_{A;x, y }(\gamma)=I_{f(A);f(x),f( y) }(f(\gamma))$ by conformal invariance. And $M:=\psi_2 \circ f \circ\psi_1^{-1}$ is a \Mo map, as it is an automorphism of $\m H$. So we have
\begin{equation*}
    \abs{M'(x)M'(y)}=\abs{\frac{M(x)-M(y)}{x-y}}^2=\frac{P_{\m H; x,y}}{P_{\m H; f(x),f(y)}}.
\end{equation*}

Using the chain rule, we have $M'\psi_1'=\psi_2'\circ ff'$ at $x$ and $y$.

Combining these, we have
\begin{align*}
    &\mc H_{\m H;f(x), f(y)}({f(\gamma)})- \mc H_{\m H;x, y}(\gamma)\\
    =& \frac{1}{12}  (I_{\m H;f(x), f(y)}(f(\gamma))-I_{\m H;x, y}(\gamma) ) -\frac{1}{4}\log\frac{ {P_{\m H; f(x),f(y)}}}{{P_{\m H; x,y}}}\\
    =& \frac{1}{4}\log\abs{f'(x)f'(y)}+ \mc B(\gamma, \m H\backslash A; \m H)- \mc B(f(\gamma), \m H\backslash f(A); \m H).
\end{align*}
\end{proof}

\subsection{Conformal restriction covariance of weighted chordal SLE }
From now on, we focus on $\kappa \in (0,4]$. Let 
\begin{equation*}
    b(\kappa)=\frac{6-\kappa}{2\kappa}
\end{equation*}
denote the boundary conformal weight. Let 
\begin{equation*}
    c(\kappa)=-\frac{(6-\kappa)(8-3\kappa)}{2\kappa}
\end{equation*}
denote the central charge.
On any simply connected domain $D$ with two distinct boundary points $x$, $y$ at which $\partial  D$ is analytic, the weighted $\SLE_\kappa$ measure on chords in $(D; x,y)$ is defined to be
\begin{equation*}
    Q^\kappa_{D; x,y}=H^\kappa_{D; x,y} \times \mu^\kappa_{D; x,y},
\end{equation*}
where $\mu^\kappa_{D; x,y}$ is the $\SLE_\kappa$ probability measure with $\sqrt{\k}$ times Brownian motion as the driving function of the Loewner equation, and $H_{D; x,y}$ is determined by the scaling rule for any conformal map $f$ on $D$
\begin{equation*}
    H^\kappa_{D; x,y}= \abs{f'(x)}^{b(\kappa)}  \abs{f'(y)}^{b(\kappa)} 
 H^\kappa_{f(D); f(x),f(y)}
\end{equation*}
and the kernel
\begin{equation*}
    H^\kappa_{\m H;x, y} =  \abs{y-x}^{-2b(\kappa)}.
\end{equation*}
We have the conformal covariance rule 
\begin{equation*}
    f \circ Q^\kappa_{D; x,y} = \abs{f'(x)}^{b(\kappa)}  \abs{f'(y)}^{b(\kappa)} Q^\kappa_{f(D); f(x),f(y)}.
\end{equation*}
Following \cite{LSW_CR_chordal,KL07}, we have
\begin{lem}\label{lem property of single chordal}
    Suppose $D \subset D' \subsetneqq \m C$ are simply connected domains. Suppose that $x$, $y$ are distinct points of $\partial D$ and $\partial D'$. And $\partial D$ and $\partial D'$ are analytic and agree in neighborhoods of $x$, $y$. Then $Q^\kappa_{D; x,y}$ is absolutely continuous with respect to $Q^\kappa_{D'; x,y}$ with Radon--Nikodym derivative 
    \begin{equation*}
        \mathbbm{1}_{\{\gamma \subset D\}} Y_{D,D';x,y}(\gamma) = \mathbbm{1}_{\{\gamma \subset D\}} \exp\left( \frac{c(\kappa)}{2} \mc B(\gamma, D'\backslash D; D')\right),
    \end{equation*}
    where $Y$ is a conformal invariant.
\end{lem}

\section{Chordal case}
\subsection{Single case}
Now, we show the uniform convergence of the total mass of Brownian loops in the upper half-plane that intersect a chord and a hull.
\begin{lem}\label{1uniform convergence}
    Let $(A_\epsilon \subset A)_{\epsilon>0}$ be a decreasing family of simply connected neighborhoods of a simple chord $\gamma$ in $(\m H;x,y)$, which agree with $\partial \m H$ in the neighborhoods of boundary points. Assume that $A_\epsilon \downarrow \gamma$ as $\epsilon\downarrow 0$, then we have the uniform convergence as follows:
\begin{equation*}
    \sup_{\{\eta \in O_\epsilon\}}  \mc B(\eta, \m H \backslash A; \m H), \inf_{\{\eta  \in O_\epsilon\}}  \mc B(\eta, \m H \backslash A; \m H)   \stackrel{\epsilon \rightarrow 0+}{\longrightarrow} \mc B(\gamma, \m H \backslash A; \m H),
\end{equation*}
where $O_\e$ denotes the set of chords in $(A_\e;x,y)$.
\end{lem}
\begin{proof}
    For the supremum, it follows from the observation that 
    \begin{equation*}
        \mc B(\gamma, \m H \backslash A; \m H) \leq \sup_{\{\eta \in O_\epsilon\}}  \mc B(\eta, \m H \backslash A; \m H)\leq \mc B(A_\e, \m H \backslash A; \m H)\rightarrow \mc B(\gamma, \m H \backslash A; \m H).
    \end{equation*}
    The inequality follows from the containment of sets, while the limit comes from monotone convergence. 
    For the infimum, let $L_\epsilon$ (or $L$) and $R_\epsilon$ (or $R$) denote the connected component of the boundary of $A_\epsilon$ (or $A$) except $\partial \m H$. Then we have
    \begin{align*}
        \mc B(\eta, \m H \backslash A; \m H)&=\mc B(\eta, L; \m H)+\mc B(\eta, R; \m H)-\mc B(L, R; \m H)\\
        &\geq \mc B(R_\epsilon, L; \m H)+\mc B(L_\epsilon, R; \m H)-\mc B(L, R; \m H)\\
        &\stackrel{\epsilon \rightarrow 0+}{\longrightarrow} \mc B(\gamma, L; \m H)+\mc B(\gamma, R; \m H)-\mc B(L, R; \m H)\\
        &=\mc B(\gamma, \m H \backslash A; \m H).
    \end{align*}
    \end{proof}
Using the map $z \rightarrow (z-\ii)/(z+\ii)$, we identify $(\m H;0,\infty)$ and $(\m D;-1,1)$. Let $\gamma_0=[-1,1]$ denote the hyperbolic geodesic in the unit disk $\m D$, whose chordal Loewner energy is $0$. For any $\epsilon > 0$, let $A_\epsilon$ denote a neighborhood of $\gamma_0$ in $\m D$, which we need to have two chords and part of the unit circle as its boundary and decrease together with $\epsilon$ and agree with $\partial \m D$ in the neighborhoods of $-1$ and $1$.
Below, we choose $A_\epsilon$ to be the domain in $\m D$ bounded by the two hyperbolic geodesics connecting  $\pm \exp( \ii\e)$ and $\mp \exp(- \ii\e)$.
Let $\gamma$ be an analytic chord such that $\gamma=f(\gamma_0)$ for some conformal map $f$ defined on some $A=A_{\epsilon_0}$ for some $\epsilon_0$ such that $\tilde A:=f(A)$ coincides with $\m S^1$ near $x=f(-1)$ and $y=f(1)$. For $\epsilon < \epsilon_0$, set $\tilde{A_\epsilon}=f(A_\epsilon)$.
Let us introduce the neighborhoods of $\gamma_0$ and $\gamma$ given by 
\begin{align*}
    &O_\epsilon(\gamma_0):=\{\text{simple~chords~in~}(A_\epsilon;-1,1)\},\\
    &O_\epsilon(\gamma):=\{\text{simple~chords~in~}(\tilde{A_\epsilon};x,y)\}.
\end{align*}
We call the sets of simple chords of the form $O_\epsilon(\gamma)$ as admissible neighborhoods.
\begin{thm}\label{OM for single chordal}
    Let $\kappa \leq 4$ and $Q_{\m H; x, y}^\kappa$ denote the chordal $\SLE_\k$ measure in $(\m H; x, y)$. For any analytic simple chord $\gamma$ connecting $x$ and $y$ such that $\gamma=f(\gamma_0)$ for some conformal map $f$ defined on some $A=A_{\epsilon_0}$ for some $\epsilon_0$. Defining a collection of admissible neighborhoods $(O_\epsilon(\gamma))_{0 < \epsilon \ll 1}$  as above, we have that 
    \begin{equation}
        \lim_{\epsilon \rightarrow 0} \frac{Q_{\m H; x, y}^\kappa(O_\epsilon(\gamma))}{Q_{\m H; 0, \infty}^\kappa(O_\epsilon(\gamma_0))}=\exp \left( \frac{c(\kappa)}{2}(\mc H(\gamma)-\mc H(\gamma_0))-\frac{3(6-\kappa)}{16} \log\abs{f'(0)f'(\infty)}\right).
    \end{equation}
 \end{thm}\begin{remark}
     The statement is still true for chordal SLE in other simply connected domains by conformal invariance if we choose the admissible neighborhoods to be the corresponding conformal images. Besides, it is the version of our main theorem when $n=1$.
 \end{remark}
\begin{proof}
    Fix $\epsilon_0$, we set $A=A_{\epsilon_0}$, $\tilde{A}=f(A_{\epsilon_0})$. 
    By conformal restriction and conformal covariance, we have
    \begin{align*}
        Q_{\m H; x, y}^\kappa(O_\epsilon(\gamma))&=\int \mathbbm{1}_{\{\tilde{\eta} \subset \tilde{A_\epsilon}\}} \dd Q_{\m H; x, y}^\kappa (\tilde{\eta}) = \int \mathbbm{1}_{\{\tilde{\eta} \subset \tilde{A_\epsilon}\}} \left(Y_{\tilde{A},\m H;x,y}(\tilde{\eta})\right)^{-1} \dd Q^\kappa_{\tilde{A}; x, y} (\tilde{\eta})\\
        &= \int \mathbbm{1}_{\{\eta \subset A_\epsilon\}} \left(Y_{\tilde{A},\m H;x,y}(f(\eta))\right)^{-1} \abs{f'(0)f'(\infty)}^{-b(\kappa)}\dd Q^\kappa_{A; 0, \infty} (\eta) \\
        &= \int \mathbbm{1}_{\{\eta \subset A_\epsilon\}} \frac{Y_{A, \m H ; 0 ,\infty}(\eta)}{Y_{\tilde{A},\m H;x,y}(f(\eta))} \abs{f'(0)f'(\infty)}^{-b(\kappa)}\dd Q_{\m H; 0, \infty}^\kappa (\eta).
  \end{align*}
        
  We have that 
  \begin{align*}
      &\frac{Y_{A, \m H ; 0 ,\infty}(\eta)}{(Y_{\tilde{A},\m H;x,y}(f(\eta))} \abs{f'(0)f'(\infty)}^{-b(\kappa)}\\
      &= \exp \left( \frac{c(\kappa
      )}{2} ( \mc B(\eta, \m H \backslash A; \m H) -  \mc B(f(\eta), \m H \backslash \tilde{A}; \m H) ) -b(\kappa) \log\abs{f'(0)f'(\infty)} \right) .
  \end{align*}
Using the uniform convergence from Lemma~\ref{1uniform convergence}, we have
\begin{equation*}
    \sup_{\{\eta \subset A_\epsilon\}} \left( \mc B(\eta, \m H \backslash A; \m H) -  \mc B(f(\eta), \m H \backslash \tilde{A}; \m H) \right)  \stackrel{\epsilon \rightarrow 0+}{\longrightarrow} \mc B(\gamma_0, \m H \backslash A; \m H) -  \mc B(\gamma, \m H \backslash \tilde{A}; \m H),
\end{equation*}
\begin{equation*}
    \inf_{\{\eta \subset A_\epsilon\}} \left( \mc B(\eta, \m H \backslash A; \m H) -  \mc B(f(\eta), \m H \backslash \tilde{A}; \m H) \right)  \stackrel{\epsilon \rightarrow 0+}{\longrightarrow} \mc B(\gamma_0, \m H \backslash A; \m H) -  \mc B(\gamma, \m H \backslash \tilde{A}; \m H).
\end{equation*}

Hence we have
\begin{align*}
    \frac{Q_{\m H; x, y}^\kappa(O_\epsilon(\gamma))}{Q_{\m H; 0, \infty}^\kappa(O_\epsilon(\gamma_0))}
    &=\frac{= \int \mathbbm{1}_{\{\eta \subset A_\epsilon\}} \frac{Y_{A, \m H ; 0 ,\infty}(\eta)}{Y_{\tilde{A},\m H;x,y}(f(\eta))} \abs{f'(0)f'(\infty)}^{-b(\kappa)}\dd Q_{\m H; 0, \infty}^\kappa (\eta)}{\int \mathbbm{1}_{\{\eta \subset A_\epsilon\}} \dd Q_{\m H; 0, \infty}^\kappa (\eta)}\\
    &\stackrel{\epsilon \rightarrow 0+}{\longrightarrow}
    \exp \left( \frac{c(\kappa)}{2}(\mc B(\gamma_0, \m H \backslash A; \m H) -  \mc B(\gamma, \m H \backslash \tilde{A}; \m H))-b(\kappa) \log\abs{f'(0)f'(\infty)}\right).
\end{align*}
By the conformal deformation of chordal Loewner potential, items in the above exponent 
\begin{align*}
     &=\frac{c(\kappa)}{2}(\mc H_{\m H;x, y}(\gamma)-\mc H_{\m H;0, \infty}(\gamma_0))-\frac{3(6-\kappa)}{16} \log\abs{f'(0)f'(\infty)}.
\end{align*}
\end{proof}
Now we show that the admissible neighborhoods form a neighborhood basis for the Hausdorff topology.
The Hausdorff distance $d_h$ of two compact sets $K_1$, $K_2 \subset \overline{\m D}$ is defined by 
\begin{equation*}
       d_h(K_1,K_2):=\inf \left\{ \epsilon > 0 \big\vert  K_1 \subset \bigcup_{z \in K_2} \overline{B_\epsilon(z)} ~\text{and}~ K_1 \subset \bigcup_{z \in K_2}\overline{B_\epsilon(z)} \right\},
\end{equation*}
where $B_\epsilon(z)$ denotes the ball of radius $\epsilon$ centered at $z$ in $\overline{\m D}$ with respect to the Euclidean metric. We endow the space $\SC$ of simple chords in $(\m D;-1,1)$ with the relative topology induced by $d_h$. We define the topology on the space of simple chords in $(D;x,y)$ via the pullback by a uniformizing conformal map $f:D \rightarrow \m D$. Although the metric depends on the choice of the conformal map $f$, the topology is canonical, as conformal automorphisms of $\m D$ are fractional linear functions, which are uniformly continuous on $\m D$. For any $\gamma \in \SC$ and $\epsilon > 0$,  let
\begin{equation*}
    B(\gamma,\epsilon)=\{z \in \overline{\m D}\big\vert d(z,\gamma) \leq \epsilon\}
\end{equation*}
and
\begin{equation*}
    B_h(\gamma,\epsilon)=\{\eta \in \SC\big\vert d_h(\eta,\gamma) < \epsilon\},
\end{equation*}
we have
\begin{equation*}
    B_h(\gamma,\epsilon)\subset\{\eta  \in \SC\big\vert \eta \subset B(\gamma,\epsilon)\}=:B^h(\gamma,\epsilon),
\end{equation*}
which is directly from the definition. 
\begin{thm}\label{lem Hausdorff topology single chordal}
    The admissible neighborhoods $O_\e$ form a neighborhood basis for the Hausdorff topology on $\SC$.
\end{thm}
\begin{proof}
    We firstly show that each $O_\e$ can be represented as a union of Hausdorff open sets. For each $\eta \in O_\e $, we have $\eta\subset A_\e$, hence there exist some $\d$ such that $B(\eta,\d) \subset A_\e$, hence we have $\eta \in B_h(\eta,\d) \subset B^h(\eta,\d) \subset O_\e$.It follows that $O_\e=\cup_{\eta \in O_\e}B_h(\eta,\d)$. In particular, $O_\e$ is an open set.

    Then, we show that each open set $O_h$ for the Hausdorff topology on $\SC$ can be represented as a union of admissible neighborhoods. For $\eta \in O_h$, $\O$ and $\O^*$ is the connected component of $\m D \backslash \eta$ while $f$ (or $g$) is the conformal map from the upper (or lower) half disk to $\O$ (or $\O^*$), respectively. Let $\eta^\d=f(L_\d)$ and $\eta_{\tilde\d}=g(R_{\tilde\d})$ for $\d$, ${\tilde \d} >0$, where $L_\d$ and $R_\d$ are the connected component of the boundary of $A_\d$ except $\partial \m D$ for the standard neighborhood. Let $F_{\d,\tilde\d}$ denote the doubly connected domain in $\m D$ bounded by $\eta^\d$ and $\eta_{\tilde\d}$, which is a simply connected domain in $\m C$. For small enough size $\d$ and $\tilde\d$, all chords contained in $F_{\d,\tilde\d}$ are in $O_h$. By the uniformization theorem, there exists $\e_0$ and a conformal map $G: A_{\e} \rightarrow F_{\d,\tilde\d}$ that fixes the boundary points $-1$ and $1$. Due to the compactness of $G^{-1}(\eta)$, we have that there exists $\e_\eta < \e$ such that $G^{-1}(\eta)\subset A_{\e_\g}$. Denote $\g_\eta=G(\g_0)$, we have $\eta \in O_{\e_\eta}(\g_{\eta})\subset O_{\e_0}(\g)\subset O_h$. It follows that $O_h=\cup_{\eta \in O_h}O_{\e_\eta}(\g_{\eta})$.
\end{proof}
    \begin{proof}[Proof of Lemma~\ref{lem Hausdorff topology}]
    The proof of Lemma~\ref{lem Hausdorff topology} in other cases is essentially the same except a few changes, which we will explain after we introduce admissible neighborhoods in the corresponding chapter.
\end{proof}
Though the range of $F_\k(\g)$ is not used in this paper, we discuss the chordal case here for those interested. 
\begin{lem}\label{lem arbitary choice of F}
      For fixed $\k\in (0,4]$, a chord $\g$ in $\m H$ connecting $-1$ and $1$, the range of $F_\k(\g)=-\frac{3(6-\kappa)}{16} \log\abs{f'(-1)f'(1)}$ is $\emptyset$ or $\m R$ for all conformal maps $f:A\rightarrow \tilde A$ that satisfy $f(\m S^1 \cap \m H)=\g$, where $A$ and $\tilde A$ are neighborhoods that belong to $\m H$ and their boundaries coincide with $\m R$ near $-1$ and $1$. In particular, the range is $\m R$ when $\g$ is locally conformally equivalent to $\m S^1 \cap \m H$.
  \end{lem}  
\begin{proof}
         For $0<r<1$, let $A_r:=\{z\in\m C|1/r<\abs z <r\}$ denote the centered annulus. Let us define the set
\begin{equation*}
    \K:=\left\{\exp(F_\k(\m S^1)) ~\Big|~ \exists r,\exists f:A_r \xrightarrow [\text{symmetric}]{\text{conformal}} \tilde{A}_r \text{ s.t. } \m S^1=f(\m S^1)\right\},
\end{equation*}
where ``symmetric'' means that $\a \circ f \circ \a^{-1}=f$ for the reflection $\a$ with respect to the real line $\m R$.

The composition and inverse of the conformal map correspond to the multiply and inverse of $\K$, which shows that $\K\ni 1$ is a multiplicative group. We claim $\K=(0,\infty)$ then we prove the lemma for any chord that is locally conformally equivalent to $\m S^1 \cap \m H$ using composition by $f$. We select compactly supported and symmetric Beltrami coefficients with respect to the real line so that the associated quasiconformal map $\o $ is symmetric with respect to the real line and conformal close to $\m S^1$ with $\o(\m S^1)=\m S^1$. It is known that the associated quasiconformal maps depend locally holomorphically (hence continuously) on their Beltrami coefficients. A linear combination of trigonometric and Polynomial Functions gives a positive number different from $1$ that belongs to $\K_{\m S^1}$ and we finish the proof.
\end{proof}
\subsection{Multiple case}
 For any positive number $n$, we fix an $n$-link pattern $\alpha$ and boundary points $x_1 < \dots <x_{2n}$, we denote the curve space by $\mc X_\alpha(\bar{x}) :=\mc X_\alpha(\m H; x_1, \ldots,x_{2n})$. Multi-chordal $\SLE_\kappa$ in $\m H$ is a family of  $\SLE_\kappa$ curves $\bar{\gamma}:=(\gamma_1,\ldots,\gamma_n)$ with interaction, which is the space of $n$ joint chords in $\m H$ connecting the marked boundary points $(x_{a_j},x_{b_j})_{j=1}^n$ according to the link pattern $\alpha$. From \cite{BPW,KL07}, the multi-chordal $\SLE_\kappa$ measure $Q_{\bar{x},\alpha}^\kappa$ can be obtained by weighting $n$ independent chordal $\SLE_\kappa$ measure  $(Q_{\m H; x_{a_j},x_{b_j}})_{j=1}^n$ with the Radon--Nikodym derivative 
 \begin{equation*}
     \exp \left( \frac{c(\kappa)}{2} \mc B(\bar{\gamma};\m H) \right),
 \end{equation*}
 where 
 \begin{equation*}
     \mc B(\bar{\gamma};\m H):= \int \max\left( \#\{\text{sets~hit~by~} \cdot\}-1 ,0\right) \dd \mu_{\m H}^{\textbf{BL}}(\cdot)
 \end{equation*}
 equals $0$ if $n=1$. So $ Q_{\bar{x},\alpha}^\kappa$ is the partition function times the standard probability measure of multi-chordal SLE. 
In \cite{PW24_LDP_multichordal}, the multi-chordal Loewner potential of $\bar{\gamma} \in \mc X_\alpha$ is defined by
 \begin{align}\label{eq multiple chordal potential}
     \mc H(\bar{\gamma}):&= \frac{1}{12}\sum_{j=1}^n I_{\m H; x_{a_j},x_{b_j}}(\gamma_j) + \mc B(\bar{\gamma};\m H)-\frac{1}{4}\sum_{j=1}^n\log\abs{x_{a_j}-x_{b_j}}^{-2}\\
     &= \sum_{j=1}^n \mc H(\gamma_j) + \mc B(\bar{\gamma};\m H),
 \end{align}
whose infimum $\mc M_{\m H}^\alpha(x_1, \ldots, x_{2n})$ over $\bar{\gamma} \in \mc X_\alpha(\bar{x})$ exists and is unique.
Next, we show the uniform convergence of $\mc B(\bar{\eta};\m H)$.
\begin{lem}\label{2uniform convergence}
    Let $(\bar A_\e=(A_\epsilon^j)_{j=1}^n \subset (A^j)_{j=1}^n)_{\epsilon>0}$ be a decreasing family of $n$ disjoint simply connected neighborhoods of $n$ simple chords $\bar{\g}$ in $\m H$, which agree with $\partial \m H$ in the neighborhoods of boundary points. Define
    \begin{align*}
        O_\epsilon(\bar{\gamma}):=\{\bar{\eta}=(\eta_1\ldots,\eta_n)~\big\vert ~\eta_j ~\text{is a simple~chord~in~}\tilde{A}_\epsilon^j, \forall j\}
    \end{align*}
    Assume that $A_\epsilon \downarrow \bar{\g}$ as $\epsilon\downarrow 0$, then we have the uniform convergence as follows:
\begin{equation*}
    \sup_{\{\bar{\eta} \in  O_\epsilon(\bar{\gamma})\}}  \mc B(\bar{\eta};\m H), \inf_{\{\bar{\eta} \in  O_\epsilon(\bar{\gamma})\}}  \mc B(\bar{\eta};\m H)   \stackrel{\epsilon \rightarrow 0+}{\longrightarrow} \mc B(\bar{\g};\m H),
\end{equation*}
\end{lem}
\begin{proof}
For the supremum, it follows from the observation that 
    \begin{equation*}
        \mc B(\bar\g; \m H) \leq \sup_{\{\bar\eta \in O_\epsilon(\bar\g)\}}  \mc B(\bar\eta; \m H)\leq \mc B(\bar A_\e; \m H)\rightarrow \mc B(\bar\gamma; \m H).
    \end{equation*}
    The inequality follows from the containment of sets, while the limit comes from monotone convergence. 
    Now let us focus on the infimum.
For $n=1$, the result is trivial. Using induction, we have that for each $j \in \{1,\ldots,n\}$,  
 \begin{equation*}
     \mc B(\bar{\eta};\m H)=\mc B(\bar{\eta}\backslash \eta_j;\m H)+\mc B(\eta_j,\m H\backslash\m H_j;\m H),
 \end{equation*}
where $\bar{\eta}\backslash \eta_j$ denotes the other $n-1$ chords except the chord $\eta_j$ and $\m H_j$ denotes the connected component of $\m H\backslash \cup_ {k \neq j } \eta_k$ containing the chord $\eta_j$. So it suffices to consider the converge of $\mc B(\eta_j,\m H\backslash\m H_j;\m H)$. There exists $j$ such that the other chords are on the same side of $\eta_j$. Assume the boundary of $\m H_j$ consists of part of $\partial \m H$ and $\cup_ {k \in E_j } \eta_k$, where $E_j$ is an index set. Fix this choice, we have $\mc B(\eta_j,\m H\backslash\m H_j;\m H)=\mc B(\eta_j,\cup_ {k \in E_j }\eta_k;\m H)$.
For $k \in E_j $, let $N_\e^j$ and $N_\e^k$ denote the other connected components of the boundary of $A_\e^j$ and $A_\e^k$ except $\partial \m H$ that are not in the middle of $\g_j$ and $\g_k$, respectively. 
Then we have
    \begin{align*}
        \mc B(\eta_j,\cup_ {k \in E_j }\eta_k;\m H)
        \geq \mc B(N_\e^j, \cup_ {k \in E_j }N_\e^k; \m H)
        \stackrel{\epsilon \rightarrow 0+}{\longrightarrow} \mc B(\g_j, \cup_ {k \in E_j }\g_k; \m H).
    \end{align*}
The inequality follows from the containment of sets, while the limit comes from monotone convergence. Note that $\mc B(\g_j, \cup_ {k \in E_j }\g_k; \m H)=\mc B(\g_j,\m H\backslash\m H_j^\prime;\m H)$, where $\m H_j^\prime$ denotes the connected component of $\m H\backslash \cup_ {k \neq j } \g_k$ containing the chord $\g_j$. Therefore, we have 
    \begin{equation*}
     \inf_{\{\bar{\eta} \in  O_\epsilon(\bar{\gamma})\}}  \mc B(\bar{\eta};\m H)   \stackrel{\epsilon \rightarrow 0+}{\longrightarrow}\mc B(\bar{\g}\backslash \g_j;\m H)+\mc B(\g_j,\m H\backslash\m H_j^\prime;\m H)= \mc B(\bar{\g};\m H).
\end{equation*}
\end{proof}
\begin{remark}
    Note that using $\mc B(\eta_j,\m H\backslash\m H_j;\m H) =\mc B(\bar{\eta};\m H)-\mc B(\bar{\eta}\backslash \eta_j;\m H)$ and the above lemma, we have $\mc B(\eta_j,\m H\backslash\m H_j;\m H)$ uniformly converges for each $j$. It can be proved directly using the inclusion-exclusion principle on the two sides of $\eta_j$. 
\end{remark}
Recall that, we identify $(\m H;0,\infty)$ and $(\m D;-1,1)$ using the map $z \rightarrow (z-\ii)/(z+\ii)$. Let us choose $\bar{x}=(x_1, \ldots,x_{2n})$ to be the image of $2n$-th root of unity $(e^{\ii\pi  m/n})_{1\leq m \leq 2n}$, $\a$ to be the link pattern connecting the two next to each other in order. Let us choose $\bar{\g}_0^\a$ to be the $n$ geodesic chords, which are semi-circles. For each $j\in \{1,\ldots,n\}$, set $A_\epsilon^j$ to be the image of $A_\e$ introduced in the single case under a \Mo map fixing the real line. Note that the choice influences the neighborhood but does not influence our theorem.
 For any analytic simple chords $\bar{\gamma}=(\gamma_1,\ldots,\gamma_n)$ such that for each $j\in \{1,\ldots,n\}$, $\gamma_j=f_j(\gamma_{0,j}^\alpha)$ for some conformal map $f_j:(A_{\epsilon_0}^j;x_{2j-1},x_{2j}) \rightarrow  (\tilde{A}_{\epsilon_0}^j \tilde{x}_{a_j},\tilde{x}_{b_j})$ for some $\epsilon_0$, define $\tilde{A}_\epsilon^j:=f_j({A}_\epsilon^j)$ for $\e <\e_0$.  Let $\b$ denote the link pattern connecting $\tilde{x}_{a_j}$ and $\tilde{x}_{b_j}$. When $\e$ is small enough, the neighborhoods for different $j$ do not touch each other. And we define 
 \begin{align*}
   O_\epsilon^j(\bar{\gamma}_{0}^\alpha):&=\{\eta_j~\big\vert ~\eta_j ~\text{is a simple~chord~in~}(A_\epsilon^j;x_{2j-1},x_{2j})\},\\
    O_\epsilon^j(\bar{\gamma} ):&=\{\eta_j~\big\vert ~\eta_j ~\text{is a simple~chord~in~}(\tilde{A}_\epsilon^j;\tilde{x}_{a_j},\tilde{x}_{b_j})\}\\
    O_\epsilon(\bar{\gamma}_0^\alpha):&=\{\bar{\eta}=(\eta_1\ldots,\eta_n)~\big\vert ~\eta_j \in  O_\epsilon^j(\bar{\gamma}_{0}^\alpha), \forall j\}\\
    &=\{\bar{\eta}=(\eta_1\ldots,\eta_n)~\big\vert ~\eta_j ~\text{is a simple~chord~in~}(A_\epsilon^j;x_{2j-1},x_{2j}),\forall j\},\\
    O_\epsilon(\bar{\gamma}):&=\{\bar{\eta}=(\eta_1\ldots,\eta_n)~\big\vert ~\eta_j \in  O_\epsilon^j(\bar{\gamma} ), \forall j\}\\
    &=\{\bar{\eta}=(\eta_1\ldots,\eta_n)~\big\vert ~\eta_j ~\text{is a simple~chord~in~}(\tilde{A}_\epsilon^j;\tilde{x}_{a_j},\tilde{x}_{b_j}), \forall j\}.
\end{align*}
We call the sets of simple multi-chords of the form $O_\epsilon(\gamma)$ as admissible neighborhoods.
In this case, Lemma~\ref{lem Hausdorff topology} holds from the construction by the product.
\begin{thm}\label{OM for multiple chordal}
    Let $\kappa \leq 4$ and $ Q_{\bar{x},\alpha}^\kappa$ denote the multi-chordal $\SLE_\k$ measure with the link pattern $\a$ in $(\m H; \bar{x})$, where $\bar{x}=(x_1, \ldots,x_{2n})$ are the boundary points of $\bar{\gamma}_0^\alpha$. For any analytic simple chords $\bar{\gamma}=(\gamma_1,\ldots,\gamma_n)$ with boundary points $\tilde{\bar{x}}$ and a pattern $\b$ such that for each $j\in \{1,\ldots,n\}$, $\gamma_j=f_j(\gamma_{0,j}^\alpha)$ for some conformal maps $f_j$ defined on some $A^j=A_{\epsilon_0}^j$ for some $\epsilon_0$, defining a collection of admissible neighborhoods $(O_\epsilon(\bar{\gamma}))_{0 < \epsilon \ll 1}$  as above, we have that 
    \begin{equation}
        \lim_{\epsilon \rightarrow 0} \frac{Q_{\tilde{\bar{x}},\b}^\kappa(O_\epsilon(\bar{\gamma}))}{Q_{\bar{x},\alpha}^\kappa(O_\epsilon(\bar{\gamma}_0^\alpha))}=\exp \left( \frac{c(\kappa)}{2} (\mc H(\bar{\gamma}) -\mc H(\bar{\gamma}_0^\alpha))
        -\frac{3(6-\kappa)}{16} \sum_{j=1}^n \log\abs{f_j'(x_{2j-1})f_j'(x_{2j})}\right).
    \end{equation}
\end{thm}
\begin{proof}
    It follows from the proof of  Theorem~\ref{OM for single chordal} and independence that 
    \begin{align*}
         \lim_{\epsilon \rightarrow 0} \prod_{j=1}^n\frac{Q^\kappa_{\m H;\tilde{x}_{a_j},\tilde{x}_{b_j}}(O_\epsilon^j(\bar{\gamma}))}{Q^\kappa_{\m H;x_{2j-1},x_{2j}}(O_\epsilon^j(\bar{\gamma}_{0}^\alpha))}
        =&\exp \left( \frac{c(\kappa)}{2} \sum_{j=1}^n \left( \mc H_{\m H; \tilde{x}_{a_j},\tilde{x}_{b_j}}(\gamma_j)-\mc H_{\m H; x_{2j-1},x_{2j}}(\gamma_{0,j})\right)\right)\\
         &\times \exp\left(-\frac{3(6-\kappa)}{16} \sum_{j=1}^n \log\abs{f_j'(x_{2j-1})f_j'(x_{2j})}\right).
    \end{align*}
    Note that even if the neighborhoods change, Lemma~\ref{1uniform convergence} still applies.
    From the definition of multi-chordal $\SLE$ and the uniform convergence of $\mc B(\bar{\gamma};\m H)$ from Lemma~\ref{2uniform convergence}, we have 
    \begin{align*}
         &\lim_{\epsilon \rightarrow 0} \frac{ Q_{\tilde{\bar{x}},\b}^\kappa(O_\epsilon(\bar{\gamma}))}{\prod_{j=1}^n Q^\kappa_{\m H;\tilde{x}_{a_j},\tilde{x}_{b_j}}(O_\epsilon^j(\bar{\gamma}))} =\exp \left( \frac{c(\kappa)}{2} \mc B(\bar{\gamma};\m H) \right),\\
         &\lim_{\epsilon \rightarrow 0} \frac{ Q_{\bar{x},\alpha}^\kappa(O_\epsilon(\bar{\gamma}_0^\alpha))}{\prod_{j=1}^n Q^\kappa_{\m H;x_{2j-1},x_{2j}}(O_\epsilon^j(\bar{\gamma}_{0}^\alpha))} =\exp \left( \frac{c(\kappa)}{2} \mc B(\bar{\gamma}_0^\alpha;\m H) \right).
    \end{align*}
    Combining these, we have
    \begin{align*}
        \lim_{\epsilon \rightarrow 0} \frac{Q_{\tilde{\bar{x}},\b}^\kappa(O_\epsilon(\bar{\gamma}))}{Q_{\bar{x},\alpha}^\kappa(O_\epsilon(\bar{\gamma}_0^\alpha))}
        =&\exp \left( \frac{c(\kappa)}{2} \sum_{j=1}^n \left( \mc H_{\m H; \tilde{x}_{a_j},\tilde{x}_{b_j}}(\gamma_j)-\mc H_{\m H; x_{2j-1},x_{2j}}(\gamma_{0,j})\right)\right)\\ 
        &\times \exp \left(\frac{c(\kappa)}{2}\left(\mc B(\bar{\gamma};\m H)-\mc B(\bar{\gamma}_0^\alpha;\m H) \right)\right) \\
        &\times \exp\left(-\frac{3(6-\kappa)}{16} \sum_{j=1}^n \log\abs{f_j'(x_{2j-1})f_j'(x_{2j})}\right)\\
        =&\exp \left( \frac{c(\kappa)}{2} \left( \mc H(\bar{\gamma}) -\mc H(\bar{\gamma}_0^\alpha) \right) -\frac{3(6-\kappa)}{16} \sum_{j=1}^n \log\abs{f_j'(x_{2j-1})f_j'(x_{2j})}\right).
  \end{align*}
\end{proof}
\subsection{Forced case}
Define chordal $\rho$-Loewner energy as in \cite{Kru_rho_energy} to be
\begin{equation*}
    I^{\rho}_{D;x,y}(\g)=\frac{1}{2}\int_0^\infty (\partial_t W_t-{\rho}\Re\frac{ 1}{W_t-V_t})^2 \dd t,
\end{equation*}
when $W$ is absolutely continuous and is $\infty$ otherwise, where $V_t=g_t(0)$.
Define $\rho$-Loewner potential to be
\begin{equation}\label{eq forced chordal potential}
    \mc H_{D;x,y}^\rho(\g)=\frac{1}{12}I^{\rho}_{D;x,y}(\g)-\frac{(\rho+2)(\rho+6)}{48}\log P_{D;x,y}.
\end{equation}
\begin{lem}\label{conformal deformation of rho-Loewner potential}
    For $\rho>-2$, let $\g$ be a chord in $\m H$ with finite $\rho$-Loewner energy and $A$ be a neighborhood of $\g$ in $\m H$ that agrees with $\partial \m H$ near the boundary points $x$ and $y$ of $\g$. Assume $f$ is a conformal map on $A$ such that $f(\g)$ is also a chord in $\m H$ and $f(A)$ agrees with $\partial \m H$ near the boundary points of $f(\g)$. Then we have 
    \begin{equation*}
        \mc H^\rho_{\m H;f(x), f(y)}({f(\gamma)})- \mc H^\rho_{\m H;x, y}(\gamma)= \frac{(\rho+2)(\rho+6)}{48}\log\abs{f'(x)f'(y)}+ \mc B(\gamma, \m H\backslash A; \m H)- \mc B(f(\gamma), \m H\backslash f(A); \m H).
    \end{equation*}
\end{lem}
\begin{proof}
    Let $W$ and $\tilde W$ denote the driving function of $\g$ in $(\m H;x,y)$ and $(A;x,y)$, respectively. Let $\psi:A \rightarrow \m H$ denote the conformal map normalized by $\psi(z)=z+ o(1)$ near $\infty$ Set $h_t:= \tilde{g}_t \circ \psi \circ g_t^{-1}$, where $\tilde{g}_t$ and $g_t$ are the conformal maps related with $\psi(\g[0, t])$ and $\g[0, t]$ from the Loewner equation, respectively. Assume that the half-plane capacity of $\g[0, t]$ is $a(t)$, it is not hard to derive that $\partial_t a= h_t' (W_t)^2$ and 
 \begin{equation*}
     \partial_t \tilde{W}_t= h_t'(W_t)(-3 h_t''(W_t)/  h_t'(W_t)+\partial_t {W}_t).
 \end{equation*}
It follows that 
\begin{align*}
    I^\rho_{A;x,y}(\g)&=\frac{1}{2}\int_0^{a(T)} (\partial_a \tilde{W}_{t(a)}-{\rho}\Re\frac{ 1}{\tilde W_{t(a)}-\tilde V_{t(a)}})^2 \dd a \\
    &=\frac{1}{2}\int_0^{T} (\frac{\partial_t \tilde{W}_{t}}{h_t'(W_t)}-{\rho}h_t'(W_t)\Re\frac{ 1}{h_t(W_t)-h_t(V_t)})^2 \dd t \\
    &=\frac{1}{2}\int_0^{T} (-3 h_t''(W_t)/  h_t'(W_t)+\partial_t {W}_t-{\rho}h_t'(W_t)\Re\frac{ 1}{h_t(W_t)-h_t(V_t)})^2  \dd t \\
    &=\int_0^{T} \frac{1}{2}(\partial_t {W}_t-{\rho}\Re\frac{ 1}{W_t-V_t})^2-4S h_t ({W}_t) -\partial_t G_t\dd t \\
    &=I^\rho_{\m H;x,y}(\g)-G_t|_{t=0}^T +12\mc B(\gamma, K; \m D),
\end{align*}
where
\begin{equation*}
    G_t=3\log h_t'(W_t) +\rho  \log \abs{\frac{h(W_t)-h_t(V_t)}{W_t-V_t}}+\frac{\rho(4+\rho)}{4}\log \abs{h_t'(V_t)}.
\end{equation*}

As $\m H\backslash A$ and $\m H\backslash f(A)$ are compact $\m H-$hulls at positive distance to the boundary points of $\g$ and $f(\g)$ respectively, we have
    \begin{equation*}
    I^\rho_{A;x, y }(\gamma)-I^\rho_{\m H;x, y}(\gamma) =\frac{(\rho+2)(\rho+6)}{4} \log\abs{\psi_1'(x)\psi_1'(y)} +12 \mc B(\gamma, \m H\backslash A; \m H),
\end{equation*}
where $\psi_1 : A \rightarrow \m H$ is conformal and fixes $x$ and $y$, and 
\begin{equation*}
    I^\rho_{f(A);f(x),f( y) }(f(\gamma))-I^\rho_{\m H;f(x), f(y)}(f(\gamma)) = \frac{(\rho+2)(\rho+6)}{4} \log\abs{\psi_2'(f(x))\psi_2'(f(y))} +12 \mc B(f(\gamma), \m H\backslash f(A); \m H),
\end{equation*}
where $\psi_2 : f(A) \rightarrow \m H$ is conformal and fixes $f(x)$ and $f(y)$.
Note that $I_{A;x, y }(\gamma)=I_{f(A);f(x),f( y) }(f(\gamma))$ by conformal invariance. And $M:=\psi_2 \circ f \circ\psi_1^{-1}$ is a \Mo map, as it is an automorphism of $\m H$. So we have
\begin{equation*}
    \abs{M'(x)M'(y)}=\abs{\frac{M(x)-M(y)}{x-y}}^2=\frac{P_{\m H; x,y}}{P_{\m H; f(x),f(y)}}.
\end{equation*}

Using the chain rule, we have $M'\psi_1'=\psi_2'\circ ff'$ at $x$ and $y$.

Combining these, we have
\begin{align*}
    &\mc H^\rho_{\m H;f(x), f(y)}({f(\gamma)})- \mc H^\rho_{\m H;x, y}(\gamma)\\
    =& \frac{1}{12}  (I^\rho_{\m H;f(x), f(y)}(f(\gamma))-I^\rho_{\m H;x, y}(\gamma) ) -\frac{(\rho+2)(\rho+6)}{48}\log\frac{ {P_{\m H; f(x),f(y)}}}{{P_{\m H; x,y}}}\\
    =& \frac{(\rho+2)(\rho+6)}{48}\log\abs{f'(x)f'(y)}+ \mc B(\gamma, \m H\backslash A; \m H)- \mc B(f(\gamma), \m H\backslash f(A); \m H).
\end{align*}
\end{proof}
For $\k\in(0,4]$ and $\rho >-2$, define
\begin{equation*}
    b_1=b_1(\k,\rho)=b(\k)=\frac{6-\k}{2\k};
\end{equation*}
\begin{equation*}
    b_2=b_2(\k,\rho)=\frac{\rho(\rho+4-\k)}{4\k};
\end{equation*}
\begin{equation*}
    b_3=b_3(\k,\rho)=\frac{\rho}{\k};
\end{equation*}
\begin{equation*}
    \a=\a(\k,\rho)=b_1+b_2+b_3=\frac{(\rho+2)(\rho+6-\k)}{4\k}.
\end{equation*}
On any simply connected domain $D$ with two distinct boundary points $x$, $y$ at which $\partial  D$ is analytic, the weighted $\SLE_\kappa(\rho)$ measure on chords in $(D; x,y)$ (with the force point $x$) is defined to be
\begin{equation*}
    Q^{\kappa,\rho}_{D; x,y}=H^{\kappa,\rho}_{D; x,y} \times \mu^{\kappa,\rho}_{D; x,y},
\end{equation*}
where $\mu^{\kappa,\rho}_{D; x,y}$ is the $\SLE_\kappa(\rho)$ probability measure, and $H^{\kappa,\rho}_{D; x,y}$ is determined by the scaling rule for any conformal map on $D$
\begin{equation*}
    H^{\kappa,\rho}_{D; x,y}= \abs{f'(x)f'(y)}^{\a} H^{\kappa,\rho}_{f(D); f(x),f(y)}
\end{equation*}
and the kernel

\begin{equation*}
    H^{\kappa,\rho}_{\m H;x, y} =  \abs{y-x}^{-2\a}.
\end{equation*}
We have the conformal covariance rule 
\begin{equation*}
    f \circ Q^{\kappa,\rho}_{D; x,y} = \abs{f'(x)f'(y)}^{\a} Q^{\kappa,\rho}_{f(D); f(x),f(y)}.
\end{equation*}
From \cite{LSW_CR_chordal}, $\SLE(\k,\rho)$ does not touch the real line except at $0$ when $\rho\geq\k/2-2$ and $\k \in (0,4]$. As mentioned at the end of Section~2 in \cite{KL07}, we can prove the two-sided conformal restriction using the local martingale from \cite{Dub05}
\begin{equation*}
    h_t'(W_t)^{b_1}h_t'(V_t)^{b_2}\frac{h_t(W_t)-h_t(V_t)}{W_t-V_t}^{b_3}\exp\left(c\int_0^t \Schwarzian h_s(W_s)\dd s\right).
\end{equation*}
Combining these together, we have the following.
\begin{lem}\label{lem property of forced chordal}
    Suppose $D \subset D' \subsetneqq \m C$ are simply connected domains. Suppose that $x$, $y$ are distinct points of $\partial D$ and $\partial D'$. And $\partial D$ and $\partial D'$ are analytic and agree in neighborhoods of $x$, $y$. Then $Q^{\kappa,\rho}_{D; x,y}$ is absolutely continuous with respect to $Q^{\kappa,\rho}_{D'; x,y}$ with Radon--Nikodym derivative 
    \begin{equation*}
        \mathbbm{1}_{\{\gamma \subset D\}} Y_{D,D';x,y}(\gamma) = \mathbbm{1}_{\{\gamma \subset D\}} \exp\left( \frac{c(\kappa)}{2} \mc B(\gamma, D'\backslash D; D')\right),
    \end{equation*}
    where $Y$ is a conformal invariant.
\end{lem}
Using the map $z \rightarrow (z-\ii)/(z+\ii)$, we identify $(\m H;0,\infty)$ and $(\m D;-1,1)$. Let $\gamma_0^\rho$ denote the chordal $\SLE_0(\rho)$ in the unit disk $(\m D;-1,1)$, whose chordal $\rho$-Loewner potential reaches the infimum. For any $\epsilon > 0$, let $A_\epsilon$ denote a neighborhood of $\gamma_0^\rho$ in $\m D$, which we need to have two chords and part of the unit circle as its boundary and decrease together with $\epsilon$ and agree with $\partial \m D$ in the neighborhoods of $-1$ and $1$.
Below, we choose $A_\epsilon$ to be the domain in $\m D$ bounded by the two chordal $\SLE_0(\rho)$ connecting  $\pm \exp( \ii\e)$ and $\mp \exp(- \ii\e)$.
Let $\gamma$ be an analytic chord such that $\gamma=f(\gamma_0^\rho)$ for some conformal map $f$ defined on some $A=A_{\epsilon_0}$ for some $\epsilon_0$ such that $\tilde A:=f(A)$ coincides with $\m S^1$ near $x=f(-1)$ and $y=f(1)$. For $\epsilon < \epsilon_0$, set $\tilde{A_\epsilon}=f(A_\epsilon)$.
Let us introduce the neighborhoods of $\gamma_0^\rho$ and $\gamma$ given by 
\begin{align*}
    &O_\epsilon(\gamma_0^\rho):=\{\text{simple~chords~in~}(A_\epsilon;-1,1)\},\\
    &O_\epsilon(\gamma):=\{\text{simple~chords~in~}(\tilde{A_\epsilon};x,y\}.
\end{align*}
We call the sets of simple chords of the form $O_\epsilon(\gamma)$ as admissible neighborhoods.
In this case, Lemma~\ref{lem Hausdorff topology} holds as we replace $\g_0$ by $\g_0^\rho$.
Similarly, we can prove the following.
\begin{thm}\label{OM for chordal variants}
    Let $\kappa \leq 4$ and $Q_{\m H; x, y}^{\kappa,\rho}$ denote the chordal $\SLE_\k$ measure in $(\m H; x, y)$. For any analytic simple chord $\gamma$ connecting $x$ and $y$ such that $\gamma=f(\gamma_0^\rho)$ for some conformal map $f$ defined on some $A=A_{\epsilon_0}$ for some $\epsilon_0$. Defining a collection of admissible neighborhoods $(O_\epsilon(\gamma))_{0 < \epsilon \ll 1}$  as above, we have that 
    \begin{align}
        \lim_{\epsilon \rightarrow 0} \frac{Q_{\m H; x, y}^{\kappa,\rho}(O_\epsilon(\gamma))}{Q_{\m H; 0, \infty}^{\kappa,\rho}(O_\epsilon(\gamma_0^\rho))}=\exp \left( \frac{c(\kappa)}{2}(\mc H_{\m H;x,y}^\rho(\gamma)-\mc H_{\m H;x,y}^\rho(\gamma_0))+F^{\k,\rho}(\g)\right),
    \end{align}
 with
 \begin{equation*}
     F^{\k,\rho}(\g)=\frac{(\rho+2)(3\rho\k+18\k-26\rho-108)}{192}\log\abs{f'(0)f'(\infty)}.
 \end{equation*}
 \end{thm}

\section{Radial case}
\subsection{Single case}
For an arc $\g$ in $D$ with one endpoint $x \in \partial D$ and the other point $y \in D$, we say $\g$ is an arc in $(D;x,y)$.

Let $\gamma$ be a simple arc in $(\m D;1, 0)$, which we choose to parametrize by the radial capacity seen from $0$. That is, the conformal map $g_t: \m D \backslash \gamma[0, t] \rightarrow \m D$ can be normalized with $g_t(0)=0$ and $g_t'(0)=e^{-t}$. By extension we can define a continuous function $ U:[0,\infty) \rightarrow \m R$ such that $ U(0)=0$ and $e^{\ii  U(\cdot)}=g_\cdot(\gamma(\cdot))$, which is called the radial driving function of  $\gamma$. The arc $\g$ can be recovered from $ U$ using the radial Loewner equation
\begin{equation*}
    \partial_t g_t(z)=g_t(z)\frac{e^{\ii  U(t)}+g_t(z)}{e^{\ii  U(t)}-g_t(z)},
\end{equation*}
with $g_0(z)=z$.
 When $ U$ is the Brownian motion with speed $\k$, the random arc is exactly a radial SLE$_\k$.
The radial Loewner energy of $\gamma$ in  $(\m D;1, 0)$ is defined by 
\begin{equation*}
    I^R_{ \m D;x, 0}(\gamma):=I( U):=\frac{1}{2}\int_0^\infty (\partial_t U(t))^2\dd t
\end{equation*}
when $ U$ is absolutely continuous and is $\infty$ otherwise. For any simply connected domain $D \subsetneqq \m C$ with a prime end $x$ and an interior point $y$ , using the conformal map $\psi:D \rightarrow \m D$ such that $\psi(x)=1$ and $\psi(y)=0$, we can define the radial Loewner energy of $\gamma$ in $(D;x,y)$ by
\begin{equation*}
    I^R_{D;x,y}(\gamma):=I^R_{\m D; 1,0}(\psi(\gamma)).
\end{equation*}

    Define the radial Loewner potential to be
\begin{equation}\label{eq single radial potential}
    \mc H^R_{\m D;x,0}(\gamma):=\frac{1}{12}I^R_{\m D; x,0}(\gamma).
\end{equation}

\begin{lem}\label{conformal deformation of radial Loewner energy}
    For a compact $\m D-$hull $K$ at positive distance to $0$ and $1$, the radial Loewner energy in $(\m D;1, 0)$ and in $(\m D \backslash K;1, 0 )$ differ by 
\begin{equation*}
    I^R_{\m D \backslash K;1, 0 }(\gamma)-I^R_{\m D;1, 0}(\gamma)= 3 \log \abs{\psi'(1)} -\frac{3}{2}\log \abs{\psi'(0)} +12 \mc B(\gamma, K; \m D),
\end{equation*}
where $\psi : \m D \backslash K \rightarrow \m D$ is conformal with $\psi(0)=0$ and $\psi'(0)>0$. Using this, we could obtain the conformal deformation of the radial Loewner energy as follows.
    
       Let $\g$ be an arc in $(\m D;x=1,y=0)$ with finite radial Loewner energy and $A$ be a neighborhood of $\g$ in $\m D$ that agrees with $\m D$ near the points $x$ and $y$. Assume $f$ is a conformal map on $A$ such that $f(\g)$ is also an arc in $\m D$ and $f(A)$ agrees with $\m D$ and $\partial f(A)$ agree with $\partial \m D$ near the boundary point $f(x)$ and the interior point $f(y)=0$. Then we have
    \begin{equation*}
        \mc H^R_{\m D;f(x), f(y)}({f(\gamma)})- \mc H^R_{\m D;x, y}(\gamma)= 
        \frac{1}{4}\log\abs{\frac{f'(x)}{f'(y)^{1/2}}}+\mc B(\gamma, \m D\backslash A; \m D)- \mc B(f(\gamma), \m D\backslash f(A); \m D).
    \end{equation*}

\end{lem}
\begin{proof}
    Set $\psi_t:= \tilde{g}_t \circ \psi \circ g_t^{-1}$ and choose a continuous $\phi_t$ such that $e^{ \ii\phi_t(z)}=\psi_t(e^{ \ii z})$ (which we call $\phi_t$ is the covering map of $\psi_t$), where $\tilde{g}_t$ and $g_t$ are the conformal maps related with $\psi(\g[0, t])$ and $\g[0, t]$ from the radial Loewner equation, respectively. It is easy to check that $\psi'(1)/\psi(1)=\phi_0'(0)$
    For the proof of the first statement, it suffices to show that for $T< \infty$,
    \begin{equation*}
    I^R_{\m D \backslash K;1, 0 }(\gamma[0, T])-I^R_{\m D;1, 0}(\gamma[0, T])= -3\log\phi_t'(U_t)|_{t=0}^T +12\mc B(\gamma, K; \m D)+\frac{3}{2}\log \psi_t'(0) |_{t=0}^T,
\end{equation*}
    since $\phi_{T_\e}'(U_{T_\e}) \rightarrow 1$ and $\psi_{T_\e}'(0)\rightarrow 1$ as $\e \rightarrow 0$, where $T_\e:=\inf\{t\geq 0 \vert \g(t)\cap B(0,\e)\neq \emptyset \}$, see \cite{JL18}.
 Set $U_t$ and $\tilde{U}_t$ to be the driving functions at time $t$. Assume that the radial capacity of $\g[0, t]$ is $a(t)$, it is not hard to derive that $\partial_t a=\phi_t' (U_t)^2$ and 
 \begin{equation*}
     \partial_t \tilde{U}_t=\phi_t'(U_t)(-3\phi_t''(U_t)/ \phi_t'(U_t)+\partial_t {U}_t).
 \end{equation*}
It follows that 
\begin{align*}
    I(\tilde{U})&=\frac{1}{2}\int_0^{a(T)} \partial_a \tilde{U}_{t(a)}^2 \dd a \\
    &=\frac{1}{2}\int_0^{T} \partial_t \tilde{U}_{t}^2 (\partial_t a)^{-1} \dd t \\
    &=\frac{1}{2}\int_0^{T} (-3\phi_t''(U_t)/ \phi_t'(U_t)+\partial_t {U}_t)^2  \dd t \\
    &=\int_0^{T} \frac{1}{2}(\partial_t {U}_t)^2-3\partial_t\log\phi_t'(U_t)-4S\phi_t ({U}_t)  \dd t \\
    &=I(U)-3\log\phi_t'(U_t)|_{t=0}^T +12\mc B(\gamma, K; \m D)+\frac{3}{2}\log \psi_t'(0) |_{t=0}^T,
\end{align*}
where 
\begin{equation*}
    S \phi_t=\frac{\phi_t'''}{\phi_t'}-\frac{3}{2}\frac{\phi_t''}{\phi_t'}^2
\end{equation*}
denotes the Schwarzian derivative. The last equality comes from the path decomposition of the Brownian loop measure.
    As $\m D\backslash A$ and $\m D\backslash f(A)$ are compact $\m D-$hulls at positive distance to the boundary and interior points of $\g$ and $f(\g)$ respectively, we have
    \begin{equation*}
    I^R_{A;x, y }(\gamma)-I^R_{\m D;x, y}(\gamma) =  3 \log \abs{\psi_1'(x)}-\frac{3}{2}\log \abs{\psi_1'(y)}  +12 \mc B(\gamma, \m D\backslash A; \m D),
\end{equation*}
where $\psi_1 : A \rightarrow \m D$ is conformal with $\psi_1(y)=y$ and $\psi_1'(y)>0$, and 
\begin{equation*}
    I^R_{f(A);f(x),f( y) }(f(\gamma))-I^R_{\m D;f(x), f(y)}(f(\gamma)) = 3 \log\abs{\psi_2'(f(x))} -\frac{3}{2}\log\abs{\psi_2'(f(y))} +12 \mc B(f(\gamma), \m D\backslash f(A); \m D),
\end{equation*}
where $\psi_2 : f(A) \rightarrow \m D$ is conformal with $\psi_2(f(y))=f(y)$ and $\psi_2'(f(y))>0$.
Note that $I^R_{A;x, y }(\gamma)=I^R_{f(A);f(x),f( y) }(f(\gamma))$ by conformal invariance. And $M:=\psi_2 \circ f \circ\psi_1^{-1}$ is a \Mo map with $M(y)=f(y)=0$ and  $M(\psi_1(x))=\psi_2(f(x))$, as it is an automorphism of $\m D$. So we have $M$ is a rotation.

Combining these, we have
\begin{align*}
   & =I^R_{\m D;f(x), f(y)}(f(\gamma))-I^R_{\m D;x, y}(\gamma)  \\
     =&3\log\abs{\frac{{f'(x)}}{M'(\psi_1(x))}}-\frac{3}{2}\log\abs{\frac{{f'(y)}}{M'(y)}}+ 12\mc B(\gamma, \m D\backslash A; \m D)- 12\mc B(f(\gamma), \m D\backslash f(A); \m D)\\
    =&3\log\abs{{f'(x)}}-\frac{3}{2}\log\abs{f'(y)}+ 12\mc B(\gamma, \m D\backslash A; \m D)- 12\mc B(f(\gamma), \m D\backslash f(A); \m D) .
\end{align*}
\end{proof}
 
For $\k \in (0,4]$, define
\begin{equation*}
    \tilde b(\k)=\frac{(6-\k)(\k-2)}{8\k}.
\end{equation*}
On any simply connected domain $D$ with a boundary points $x$ and an interior ponit $y$ at which $\partial  D$ is analytic, the weighted radial $\SLE_\kappa$ measure on arcs in $(D; x,y)$ is defined to be
\begin{equation*}
    Q^\kappa_{D; x,y}=H^\kappa_{D; x,y} \times \mu^\kappa_{D; x,y},
\end{equation*}
where $\mu^\kappa_{D; x,y}$ is the radial $\SLE_\kappa$ probability measure with $\sqrt{\k}$ times Brownian motion as the driving function of the radial Loewner equation, and $H^\kappa_{D; x,y}$ is determined by the scaling rule for any conformal map on $D$
\begin{equation*}
    H^\kappa_{D; x,y}= \abs{f'(x)}^{b(\kappa)}  \abs{f'(y)}^{\tilde b(\kappa)} 
 H^\kappa_{f(D); f(x),f(y)}
\end{equation*}
and the kernel
\begin{equation*}
    H^\kappa_{\m D;1, 0} = 1.
\end{equation*}
We have the conformal covariance rule
\begin{equation*}
    f \circ Q^\kappa_{D; x,y} = \abs{f'(x)}^{b(\kappa)}  \abs{f'(y)}^{\tilde b(\kappa)} Q^\kappa_{f(D); f(x),f(y)}.
\end{equation*}
Following \cite{Wu15_conformal_restriction_radial,JL18}, we have
\begin{lem}\label{lem property of single radial}
    Suppose $D \subset D' \subsetneqq \m C$ are simply connected domains. Suppose that $x$ is a boundary point of $\partial D$ and $\partial D'$ and $y$ is an interior point of $D$ and $D'$. Besides, $\partial D$ and $\partial D'$ are analytic and agree in neighborhoods of $x$ while $ D$ and $ D'$ are analytic and agree in neighborhoods of $y$.  Then $Q^\kappa_{D; x,y}$ is absolutely continuous with respect to $Q^\kappa_{D'; x,y}$ with Radon--Nikodym derivative 
    \begin{equation*}
        \mathbbm{1}_{\{\gamma \subset D\}} Y_{D,D';x,y}(\gamma) = \mathbbm{1}_{\{\gamma \subset D\}} \exp\left( \frac{c(\kappa)}{2} \mc B(\gamma, D'\backslash D; D')\right),
    \end{equation*}
    where $Y$ is a conformal invariant.
\end{lem}
Now, we show the uniform convergence of the total mass of Brownian loops in the unit disk that intersect an arc and a hull.
\begin{lem}\label{3 uniform convergence}
    Let $(A_\epsilon \subset A)_{\epsilon>0}$ be a decreasing family of simply connected neighborhoods of a multi-arc $\gamma$ in $(\m D;\bar x,y)$, which agree with $\m D$ or $\partial \m D$ in the neighborhoods of the interior point $y$ or the boundary point $\bar x$, respectively. Assume that $A_\epsilon \downarrow \gamma$ as $\epsilon\downarrow 0$ and let $O_\e$ denote the set of multi-arcs in $(A_\e;\bar x,y)$, we assume
    \begin{equation*}
        \sup_{\eta \in O_\e} d_h(\eta,\g) \stackrel{\epsilon \rightarrow 0+}{\longrightarrow} 0.
    \end{equation*}
    Then we have the uniform convergence as follows:
\begin{equation*}
    \sup_{\{\eta  \in O_\epsilon\}}  \mc B(\eta, \m D \backslash A; \m D), \inf_{\{\eta  \in O_\epsilon\}}  \mc B(\eta, \m D \backslash A; \m D)   \stackrel{\epsilon \rightarrow 0+}{\longrightarrow} \mc B(\gamma, \m D \backslash A; \m D).
\end{equation*}
\end{lem}
\begin{proof}
  The uniform convergence comes from the continuity of Brownian loop measure with respect to the Hausdorff distance.
    \end{proof}
Let $\g^0=[0,1]$ denote the hyperbolic geodesic in the unit disk $\m D$, whose radial Loewner energy is $0$. For any $\epsilon > 0$, let $A_\epsilon$ denote a neighborhood of $\g^0$ in $\m D$, which we need to have one chord and part of the unit circle as its boundary and decrease together with $\epsilon$ and agree with $\m D$ or $\partial \m D$ in the neighborhoods of $0$ or $1$, respectively.
Below, we choose $A_\epsilon$ such that the chord is chosen to be the union of the left semi-circle centered at $0$ with radius $\e$ and two horizontal lines connecting the unit circle.
Let $\gamma$ be an analytic arc in $(\m D;x\in \partial \m D,0)$ such that $\gamma=f(\g^0)$ for some conformal map $f$ defined on some $A=A_{\epsilon_0}$ for some $\epsilon_0$ with $f(0)=0$, $f(1)=x$ and $\abs{f'(0)f'(1)^6}=1$. Set $\tilde{A}=f(A)$. For $\epsilon < \epsilon_0$, set $\tilde{A_\epsilon}=f(A_\epsilon)$.
Let us introduce the neighborhoods of $\g^0$ and $\gamma$ given by 
\begin{align*}
    &O_\epsilon(\g^0):=\{\text{simple~arcs~in~}(A_\epsilon;1,0)\},\\
    &O_\epsilon(\gamma):=\{\text{simple~arcs~in~}(\tilde{A_\epsilon};x,0)\}.
\end{align*}
We call the sets of simple arcs of the form $O_\epsilon(\gamma)$ as admissible neighborhoods.
In this case, Lemma~\ref{lem Hausdorff topology} holds by firstly extending the arc $\g$ to a chord by the hyperbolic geodesic in $\m D\backslash \g$ connecting the interior point $0$ and the reflection $-x$ of the boundary point $x$ with respect to the interior point, and then restricting the neighborhoods. We do this to modify the previous proof, otherwise we can just write a separate proof.

\begin{thm}\label{OM for single radial}
    Let $\kappa \leq 4$ and $Q_{\m D; x, y}^\kappa$ denote the radial $\SLE$ measure in $(\m D; x, y=0)$. For any analytic simple arc $\gamma$ connecting a boundary point $x$ and an interior point $y$ such that $\gamma=f(\g^0)$ for some conformal map $f$ defined on some $A=A_{\epsilon_0}$ for some $\epsilon_0$ and define a collection of admissible neighborhoods $(O_\epsilon(\gamma))_{0 < \epsilon \ll 1}$  as above, we have that 
 \begin{equation}
        \lim_{\epsilon \rightarrow 0} \frac{Q_{\m D; x, y}^\kappa(O_\epsilon(\gamma))}{Q_{\m D; 1, 0}^\kappa(O_\epsilon(\g^0))}=\exp \left( \frac{c(\kappa)}{2}(\mc H^R(\gamma)-\mc H^R(\g^0))
        \right).
    \end{equation}
 \end{thm}
\begin{proof}
    Fix $\epsilon_0$, we set $A=A_{\epsilon_0}$, $\tilde{A}=f(A_{\epsilon_0})$. 
    By conformal restriction and conformal covariance, we have
    \begin{align*}
        Q_{\m D; x, y}^\kappa(O_\epsilon(\gamma))&=\int \mathbbm{1}_{\{\tilde{\eta} \subset \tilde{A_\epsilon}\}} \dd Q_{\m D; x, y}^\kappa (\tilde{\eta}) = \int \mathbbm{1}_{\{\tilde{\eta} \subset \tilde{A_\epsilon}\}} \left(Y_{\tilde{A},\m D;x,y}(\tilde{\eta})\right)^{-1} \dd Q^\kappa_{\tilde{A}; x, y} (\tilde{\eta})\\
        &= \int \mathbbm{1}_{\{\eta \subset A_\epsilon\}} \left(Y_{\tilde{A},\m D;x,y}(f(\eta))\right)^{-1} \abs{f'(1)}^{-b(\kappa)}\abs{f'(0)}^{-\tilde{b}(\kappa)}\dd Q^\kappa_{A; 1, 0} (\eta) \\
        &= \int \mathbbm{1}_{\{\eta \subset A_\epsilon\}} \frac{Y_{A, \m D ; 1, 0}(\eta)}{Y_{\tilde{A},\m D;x,y}(f(\eta))}\abs{f'(1)}^{-b(\kappa)}\abs{f'(0)}^{-\tilde{b}(\kappa)}\dd Q_{\m D; 1, 0}^\kappa (\eta).
  \end{align*}
        
  We have that 
  \begin{align*}
      &\frac{Y_{A, \m D ; 1, 0}(\eta)}{(Y_{\tilde{A},\m D;x,y}(f(\eta))} \abs{f'(1)}^{-b(\kappa)}\abs{f'(0)}^{-\tilde{b}(\kappa)}\\
      &= \exp \left( \frac{c(\kappa
      )}{2} ( \mc B(\eta, \m D \backslash A; \m D) -  \mc B(f(\eta), \m D \backslash \tilde{A}; \m D) )-b(\kappa) \log\abs{f'(1)}-\tilde{b}(\kappa)\log\abs{f'(0)} \right) .
  \end{align*}
Using the uniform convergence from Lemma~\ref{3 uniform convergence}, we have
\begin{equation*}
    \sup_{\{\eta \subset A_\epsilon\}} \left( \mc B(\eta, \m D \backslash A; \m D) -  \mc B(f(\eta), \m D \backslash \tilde{A}; \m D) \right)  \stackrel{\epsilon \rightarrow 0+}{\longrightarrow} \mc B(\g^0, \m D \backslash A; \m D) -  \mc B(\gamma, \m D \backslash \tilde{A}; \m D),
\end{equation*}
\begin{equation*}
    \inf_{\{\eta \subset A_\epsilon\}} \left( \mc B(\eta, \m D \backslash A; \m D) -  \mc B(f(\eta), \m D \backslash \tilde{A}; \m D) \right)  \stackrel{\epsilon \rightarrow 0+}{\longrightarrow} \mc B(\g^0, \m D \backslash A; \m D) -  \mc B(\gamma, \m D \backslash \tilde{A}; \m D).
\end{equation*}

Hence we have
\begin{align*}
    &\frac{Q_{\m D; x, y}^\kappa(O_\epsilon(\gamma))}{Q_{\m D; 1, 0}^\kappa(O_\epsilon(\g^0))}
    =\frac{ \int \mathbbm{1}_{\{\eta \subset A_\epsilon\}} \frac{Y_{A, \m D ; 1, 0}(\eta)}{Y_{\tilde{A},\m D;x,y}(f(\eta))} \abs{f'(1)}^{-b(\kappa)}\abs{f'(0)}^{-\tilde{b}(\kappa)}\dd Q_{\m D; 1, 0}^\kappa (\eta)}{\int \mathbbm{1}_{\{\eta \subset A_\epsilon\}} \dd Q_{\m D; 1, 0}^\kappa (\eta)}\\
    &\stackrel{\epsilon \rightarrow 0+}{\longrightarrow}
    \exp \left( \frac{c(\kappa)}{2}(\mc B(\g^0, \m D \backslash A; \m D) -  \mc B(\gamma, \m D \backslash \tilde{A}; \m D))-b(\kappa) \log\abs{f'(1)}-\tilde{b}(\kappa)\log\abs{f'(0)}\right).
\end{align*}
By the conformal deformation of radial Loewner energy, items in the above exponent 
\begin{align*}
     &=\frac{c(\kappa)}{24}(I^R_{\m D;x, y}(\gamma)-I^R_{\m D;1, 0}(\g^0))-\frac{3(6-\kappa)}{16} \log\abs{f'(1)}-\frac{(6-\kappa)}{32} \log\abs{f'(0)}.
\end{align*}
\end{proof}
\subsection{Forced case}
For $\k\in(0,4]$ and $\rho >-2$, define
\begin{equation*}
    b_1=b_1(\k,\rho)=b(\k)=\frac{6-\k}{2\k};
\end{equation*}
\begin{equation*}
    b_2=b_2(\k,\rho)=\frac{\rho(\rho+4-\k)}{4\k};
\end{equation*}
\begin{equation*}
    b_3=b_3(\k,\rho)=\frac{\rho}{\k};
\end{equation*}
\begin{equation*}
    \a=\a(\k,\rho)=b_1+b_2+b_3=\frac{(\rho+2)(\rho+6-\k)}{4\k};
\end{equation*}
\begin{equation*}
    \b=\b(\k,\rho)=\frac{(\rho+\k-2)(\rho+6-\k)}{8\k};
\end{equation*}
On any simply connected domain $D$ with a boundary points $x$ at which $\partial  D$ is analytic and an interior point $y$, the weighted radial $\SLE_\kappa(\rho)$ measure on arcs in $(D; x,y)$ (with the force point $x$) is defined to be
\begin{equation*}
    Q^{\kappa,\rho}_{D; x,y}=H^{\kappa,\rho}_{D; x,y} \times \mu^{\kappa,\rho}_{D; x,y},
\end{equation*}
where $\mu^{\kappa,\rho}_{D; x,y}$ is the radial $\SLE_\kappa(\rho)$ probability measure, and $H^{\kappa,\rho}_{D; x,y}$ is determined by the scaling rule for any conformal map on $D$
\begin{equation*}
    H^{\kappa,\rho}_{D; x,y}= \abs{f'(x)}^\a\abs{f'(y)}^{\b} H^{\kappa,\rho}_{f(D); f(x),f(y)}
\end{equation*}
and the kernel

\begin{equation*}
    H^{\kappa,\rho}_{\m D;1, 0} = 1.
\end{equation*}
We have the conformal covariance rule 
\begin{equation*}
    f \circ Q^{\kappa,\rho}_{D; x,y} = \abs{f'(x)}^\a\abs{f'(y)}^{\b} Q^{\kappa,\rho}_{f(D); f(x),f(y)}.
\end{equation*}
Similarly as in the chordal case, when $\rho\geq\k/2-2$ and $\k \in (0,4]$, we can prove the two-sided conformal restriction.
\begin{lem}\label{lem property of forced radial}
    Suppose $D \subset D' \subsetneqq \m C$ are simply connected domains. Suppose that $x$ is a boundary point of $\partial D$ and $\partial D'$ and $y$ is an interior point of $D$ and $D'$. Besides, $\partial D$ and $\partial D'$ are analytic and agree in neighborhoods of $x$ while $ D$ and $ D'$ are analytic and agree in neighborhoods of $y$. Then $Q^{\kappa,\rho}_{D; x,y}$ is absolutely continuous with respect to $Q^{\kappa,\rho}_{D'; x,y}$ with Radon--Nikodym derivative 
    \begin{equation*}
        \mathbbm{1}_{\{\gamma \subset D\}} Y_{D,D';x,y}(\gamma) = \mathbbm{1}_{\{\gamma \subset D\}} \exp\left( \frac{c(\kappa)}{2} \mc B(\gamma, D'\backslash D; D')\right),
    \end{equation*}
    where $Y$ is a conformal invariant.
\end{lem}
Define radial $\rho$-Loewner energy as in \cite{Kru_rho_energy} to be
\begin{equation*}
    I^{R,\rho}_{D;x,y}(\g)=\frac{1}{2}\int_0^\infty (\partial_t W_t-\frac{\rho}{2}\cot\frac{ W_t-V_t}{2})^2 \dd t,
\end{equation*}
when $W$ is absolutely continuous and is $\infty$ otherwise, where $V:[0,\infty)\rightarrow \m R$ is a continuous function determined by $V(0)=0$ and $e^{\ii V_t}=g_t(1)$.
Define the radial $\rho$-Loewner potential to be
\begin{equation}\label{eq forced radial potential}
    \mc H^{R,\rho}_{\m D;x,0}(\gamma):=\frac{1}{12}I^{R,\rho}_{\m D; x,0}(\gamma).
\end{equation}
Let $\gamma^0_\rho$ denote the radial $\SLE_0(\rho)$ in the unit disk $(\m D;1,0)$, whose radial $\rho$-Loewner energy reaches the infimum. For any $\epsilon > 0$, let $A_\epsilon$ denote a neighborhood of $\gamma^0_\rho$ in $\m D$, which we need to have one chord and part of the unit circle as its boundary and decrease together with $\epsilon$ and agree with $\m D$ or $\partial \m D$ in the neighborhoods of $0$ or $1$, respectively.
Below, we choose $A_\epsilon$ to be the domain in $\m D$ with distance from $\gamma^0_\rho$ less than $\e$.
Let $\gamma$ be an analytic arc in $(\m D;x\in \partial \m D,0)$ such that $\gamma=f(\gamma^0_\rho)$ for some conformal map $f$ defined on some $A=A_{\epsilon_0}$ for some $\epsilon_0$ with $f(0)=0$ and $f(1)=x$. Set $\tilde{A}=f(A)$. For $\epsilon < \epsilon_0$, set $\tilde{A_\epsilon}=f(A_\epsilon)$.
Let us introduce the neighborhoods of $\g^0$ and $\gamma$ given by 
\begin{align*}
    &O_\epsilon(\g^0):=\{\text{simple~arcs~in~}(A_\epsilon;1,0)\},\\
    &O_\epsilon(\gamma):=\{\text{simple~arcs~in~}(\tilde{A_\epsilon};x,0)\}.
\end{align*}
We call the sets of simple arcs of the form $O_\epsilon(\gamma)$ as admissible neighborhoods.
In this case, Lemma~\ref{lem Hausdorff topology} holds by replacing $\g^0$ by $\g^0_\rho$ compared with the single radial case.
Similarly, we can prove the following.
\begin{thm}\label{OM for radial variants}
    Let $\kappa \leq 4$ and $Q_{\m D; x, y}^{\kappa,\rho}$ denote the radial $\SLE_\k$ measure in $(\m D; x, y)$. For any analytic simple arc $\gamma$ connecting $x$ and $y$ such that $\gamma=f(\gamma_0^\rho)$ for some conformal map $f$ defined on some $A=A_{\epsilon_0}$ for some $\epsilon_0$ and define a collection of admissible neighborhoods $(O_\epsilon(\gamma))_{0 < \epsilon \ll 1}$  as above, we have that 
    \begin{align}
        \lim_{\epsilon \rightarrow 0} \frac{Q_{\m D; x, 0}^{\kappa,\rho}(O_\epsilon(\gamma))}{Q_{\m D; 1,0}^{\kappa,\rho}(O_\epsilon(\gamma^0_\rho))}=\exp \left( \frac{c(\kappa)}{2}(\mc H^{R,\rho}_{D;x,y}(\gamma)-\mc H^{R,\rho}_{D;x,y}(\gamma^0_\rho))+F^{\k,\rho}(\g)\right),
    \end{align}
 with 
 \begin{align*}
     F^{\k,\rho}(\g)=&\frac{(\rho+2)(3\rho\k+18\k-26\rho-108)}{192}\log\abs{f'(1)}\\
     &+\frac{3(\rho+2)^2\k-2(13\rho^2+52\rho+36)}{384}\log\abs{f'(0)}.
 \end{align*}
 \end{thm}
 \subsection{Multiple case}
For $\k\in(0,4]$, an integer $n\geq2$ and $\mu \in \m R$, define
\begin{equation*}
    \tilde b_n=\tilde b_n(\k,\mu)=\frac{n^2-1-\mu^2}{2\k}.
\end{equation*}
On any simply connected domain $D$ with $n$ boundary points $\bar x=(x_1,\ldots,x_n)$ at which $\partial  D$ is analytic and an interior point $y$, the weighted multi-radial $\SLE_\kappa$ measure with spiraling rate $\mu$ on multi-arcs in $(D; \bar x,y)$ is defined to be
\begin{equation*}
    Q^{n,\kappa,\mu}_{D; \bar x,y}=H^{n,\kappa,\mu}_{D; \bar x,y} \times \mu^{n,\kappa,\mu}_{D;\bar x,y},
\end{equation*}
where $\mu^{n,\kappa,\mu}_{D;\bar x,y}$ is the multi-radial $\SLE_\kappa$ probability measure with spiraling rate $\mu$, and $H^{n,\kappa,\mu}_{D; \bar x,y}$ is determined by the scaling rule for any conformal map on $D$
\begin{equation*}
    H^{n,\kappa,\mu}_{D; \bar x,y}= \prod_{j=1}^n\abs{f'(x_j)}^b\abs{f'(y)}^{\tilde b+\tilde b_n} H^{n,\kappa,\mu}_{f(D);  f(\bar x),f(y)}
\end{equation*}
and the kernel

\begin{equation*}
    H^{n,\kappa,\mu}_{\m D;\bar x, 0} = \prod_{1\leq j< \ell \leq n}\left(\sin\frac{\t_j-\t_\ell}{2}\right)^{2/\k}\times \exp\left(\frac{\mu}{\k}\sum_{j=1}^n\t_j\right),
\end{equation*}
where $\bar x=(e^{2\pi\ii \t_j})_{j=1}^n$.
We have the conformal covariance rule 
\begin{equation*}
    f \circ Q^{n,\kappa,\mu}_{D;\bar x,y} = \prod_{j=1}^n\abs{f'(x_j)}^b\abs{f'(y)}^{\tilde b+\tilde b_n} Q^{n,\kappa,\mu}_{f(D); f(\bar x),f(y)}.
\end{equation*}
From \cite{HPW_multiradial_perturbation}, we have the following conformal restriction.
\begin{lem}\label{lem property of multiple radial}
    Suppose $D \subset D' \subsetneqq \m C$ are simply connected domains. Suppose that $\bar x$ are boundary points of $\partial D$ and $\partial D'$ and $y$ is an interior point of $D$ and $D'$. Besides, $\partial D$ and $\partial D'$ are analytic and agree in neighborhoods of $\bar x$ while $ D$ and $ D'$ are analytic and agree in neighborhoods of $y$. Then $Q^{n,\kappa,\mu}_{D; \bar x,y}$ is absolutely continuous with respect to $Q^{n,\kappa,\mu}_{D'; \bar x,y}$ with Radon--Nikodym derivative
    \begin{equation*}
        \mathbbm{1}_{\{\gamma \subset D\}} Y_{D,D';x,y}(\gamma) = \mathbbm{1}_{\{\gamma \subset D\}} \exp\left( \frac{c(\kappa)}{2} \mc B(\gamma, D'\backslash D; D')\right),
    \end{equation*}
    where $Y$ is a conformal invariant. 
\end{lem}
Define the multi-radial Loewner energy as in \cite{HPW_multiradial_LDP}.
For $\bar x^{0}=(e^{2\pi \ii m/n})_{m=1}^n$, let $\gamma^{0,\mu}_n$ denote the multi-radial $\SLE_0$ with spiraling rate $\mu$ in the unit disk $(\m D;\bar x^0,0)$, whose multi-radial Loewner energy reaches the infimum. For any $\epsilon > 0$, let $A_\epsilon$ denote a neighborhood of $\gamma^0_n,\mu$ in $\m D$, which we need to have $n$ chords and part of the unit circle as its boundary and decrease together with $\epsilon$ and agree with $\m D$ or $\partial \m D$ in the neighborhoods of $\bar x^0$ or $1$, respectively.
Below, we choose $A_\epsilon$ to be the domain in $\m D$ with distance from $\gamma^{0,\mu}_n$ less than $\e$.
Let $\gamma$ be an analytic arc in $(\m D;\bar x,0)$ such that $\gamma=f(\gamma^{0,\mu}_n)$ for some conformal map $f$ defined on some $A=A_{\epsilon_0}$ for some $\epsilon_0$ with $f(0)=0$, $f(\bar x^0)=\bar x$. Set $\tilde{A}=f(A)$. For $\epsilon < \epsilon_0$, set $\tilde{A_\epsilon}=f(A_\epsilon)$.
Let us introduce the neighborhoods of $\g^{0,\mu}_n$ and $\gamma$ given by 
\begin{align*}
    &O_\epsilon(\g^{0,\mu}_n):=\{\text{simple~multi-arcs~in~}(A_\epsilon;\bar x^0,0)\},\\
    &O_\epsilon(\gamma):=\{\text{simple~multi-arcs~in~}(\tilde{A_\epsilon};\bar x,0)\}.
\end{align*}
We call the sets of simple multi-arcs of the form $O_\epsilon(\gamma)$ as admissible neighborhoods.
In this case, Lemma~\ref{lem Hausdorff topology} holds as follows. As the complement of the multi-arc in the unit disk contains $n$ connected components rather than two connected components, we need to choose appropriate size $\d$ of each part of the neighborhood such that the $\g_\eta=G(\g_0)$ is a multi-arc connecting $0$ and $n$ boundary points on the unit circle, but the rest is the same.

Our assumption of the existence of conformal maps requires that the tangent angles at $0$ of multi-arcs are the same. The characterization of finite-energy multi-arcs in \cite{AOP_LDP_multiradial} is that the driving functions eventually approach equal spacing around the circle. However, for the need of neighborhood basis, we also introduce analytic multi-arcs with with different tangent angles at $0$, which could be locally conformally equivalent to a the union of $n$ hyperbolic geodesic connecting $0$ and each boundary point. Then we define the admissible neighborhoods similarly using this conformal map.

For $n$ arcs in $\m D$ with connecting $\bar x=(e^{\ii \t_j})_{j=1}^n$ and $0$ and a spiral rate $\mu$, define the multi-radial Loewner potential to be
\begin{align}\label{eq multiple radial potential}
    \mc H^{n,R,\mu}_{\m D;\bar x, 0}( \gamma)=&\frac{1}{12}\mc J^{n,R,\mu}_{\m D;\bar x, 0}( \g)-\frac{1}{6}\sum_{1\leq j <\ell \leq n} \log\frac{\sin(\t_j-\t_i)}{2}-\frac{\mu}{12}\sum_{j=1}^n \t_j,
\end{align}
where $\mc J^{n,R,\mu}_{\m D;\bar x, 0}$ denotes the multi-radial Loewner energy defined in \cite{AOP_LDP_multiradial,HPW_multiradial_LDP}.
Similarly, using the calculation in the boundary perturbation, we can prove the following.

\begin{lem}\label{conformal deformation of multi-radial Loewner potentail}
       Let $\bar\g$ be a multi-arc in $(\m D;\bar x,y=0)$ with finite multi-radial Loewner energy and $A$ be a neighborhood of $\bar\g$ in $\m D$ that agrees with $\m D$ near the points $\bar x$ and $y$. Assume $f$ is a conformal map on $A$ such that $f(\g)$ is also an arc in $\m D$ and $f(A)$ agrees with $\m D$ and $\partial f(A)$ agree with $\partial \m D$ near the boundary point $f(\bar x)$ and the interior point $f(y)=0$. Then we have

    \begin{align*}
        \mc H^{n,R,\mu}_{\m D;f(\bar x), f(y)}({f(\gamma)})- \mc H^{n,R,\mu}_{\m D;\bar x, y}(\gamma)=& \mc B(\gamma, \m D\backslash A; \m D)- \mc B(f(\gamma), \m D\backslash f(A); \m D)\\
        &+\frac{1}{4}\sum_{j=1}^n\log \abs{f'(x_j)} +\frac{n^2-4-\mu^2}{24}\log \abs{f'(0)}.
    \end{align*}
    
\end{lem}
\begin{proof}
    For $\bar t=(t_j)_{j=1}^n$ and the arcs $\g=(\g_j)_{j=1}^n$ connecting $0$ and $(x_j)_{j=1}^n$, using the conformal map $g$ below normalized at the origin with $g(0)=0$ and $g'(0)>0$:
    \begin{itemize}
        \item $g_{ t_j}^{(j)}:\m D \backslash \g^j([0,t_j])\rightarrow \m D$ for $1 \leq j\leq n$;
        \item $g_{\bar t}:\m D \backslash \cup_{j=1}^n\g_j([0,t_j])\rightarrow \m D$;
        \item $g_{\bar t,j}:\m D \backslash g_{ t_j}^{(j)}\left(\cup_{k\neq j}\g^k([0,t_k])\right)\rightarrow \m D$ for $1 \leq j\leq n$.
    \end{itemize}
    They are related such that $g_{\bar t}=g_{\bar t,j}\circ g_{ t_j}^{(j)}$. Let $\phi_{ t_j}^{(j)}$, $\phi_{\bar t}$, $\phi_{\bar t,j}$ be the covering maps of $g_{ t_j}^{(j)}$, $g_{\bar t}$, $g_{\bar t,j}$ respectively. Denote by $U_{t_j}^{j}$ the radial driving function of $\g^j$. Define the multi-time driving function of $\g_{\bar t}$ by $\t_j(\bar t)=\phi_{\bar t,j}(U_{t_j}^{j})$.
    From \cite{HPW_multiradial_LDP}, we have
    \begin{align*}
        \mc H^{n,R,\mu}_{\m D;\bar x, y}( \g[0,\bar T])=&\sum_{j=1}^n \mc H_{\m D;x_j,y}^R(\g_j[0,T_j])+ m_{\bar T}(\g) \\
        &- \frac{n^2+3n-4-\mu^2}{24}\log g_{\bar T}'(0)-\frac{1}{4}\sum_{j=1}^n\left(\log \phi_{\bar T,j}'\left(U_{t_j}^{(j)}\right)-\frac{1}{2}\log g_{\bar T,j}'(0)\right)\\
        &-\frac{1}{6}\sum_{1\leq j <\ell \leq n} \log\frac{\sin(\t_j(\bar T)-\t_\ell(\bar T))}{2}-\frac{\mu}{12}\sum_{j=1}^n \t_j(\bar T),
    \end{align*}
    where $m_{\bar t}(\g)$ is a Brownian loop term defined with $m_{\bar 0}=0$ and using
    \begin{equation*}
        \dd m_{\bar t}(\g) =\sum_{j=1}^n \left(-\frac{1}{3}\Schwarzian \phi_{\bar t,j}\left(U_{t_j}^{(j)}\right)+\frac{1}{6}\left(1-\phi_{\bar t,j}'\left(U_{t_j}^{(j)}\right)\right)\right)\dd t_j.
    \end{equation*}
    Using the similar formula for $f(\g[0,\bar T])$, compare the difference and let $\bar T \to \infty$, we have
    \begin{align*}
        \mc H^{n,R,\mu}_{\m D;f(\bar x), f(y)}({f(\gamma)})- \mc H^{n,R,\mu}_{\m D;\bar x, y}(\gamma)=& \mc B(\gamma, \m D\backslash A; \m D)- \mc B(f(\gamma), \m D\backslash f(A); \m D)\\
        &+\frac{1}{4}\sum_{j=1}^n\log \abs{f'(x_j)} +\frac{n^2-4-\mu^2}{24}\log \abs{f'(0)}.
    \end{align*}
\end{proof}
Using Lemma~\ref{conformal deformation of multi-radial Loewner potentail}, we have the following theorem.
\begin{thm}\label{OM for multiple radial}
    Let $\kappa \leq 4$, $n\geq 2$ be an integer and $Q_{\m D; x, y}^{n,\kappa,\mu}$ denote the multi-radial $\SLE_\k$ measure with spiraling rate $\mu$ in $(\m D; \bar x, y)$. For any analytic simple multi-arcs $\gamma$ connecting $\bar x$ and $y=0$ such that $\gamma=f(\gamma^{0,\mu}_n)$ for some conformal map $f$ defined on some $A=A_{\epsilon_0}$ for some $\epsilon_0$ and define a collection of admissible neighborhoods $(O_\epsilon(\gamma))_{0 < \epsilon \ll 1}$  as above, we have that
    \begin{align}
        \lim_{\epsilon \rightarrow 0} \frac{Q_{\m D; \bar x, y}^{n,\kappa,\mu}(O_\epsilon(\gamma))}{Q_{\m D; \bar x^0,y}^{n,\kappa,\mu}(O_\epsilon(\gamma^{0,\mu}_n))}=\exp \left( \frac{c(\kappa)}{2}(\mc H^{n,R,\mu}_{D;\bar x,y}(\gamma)-\mc H^{n,R,\mu}_{D;\bar x^{0},y}(\gamma^0_n))+F^{n,\k}(\g)\right),
    \end{align}
 with
 \begin{equation*}
     F^{n,\k}(\g)=-\frac{3(6-\kappa)}{16}\sum_{j=1}^\infty\log\abs{f'(x^0_j)}+\left(\frac{(3\k-26)(n^2-1-\mu^2)}{96}-\frac{6-\k}{32}\right)\log\abs{f'(0)}.
 \end{equation*}
 \end{thm}

\bibliographystyle{plain}
\bibliography{ref}

\end{document}